\documentclass[11pt,reqno]{amsart}
\title[Quandles as pre-Lie skew braces, Hopf algebras $\&$ universal ${\cal R}$-matrices]
{Quandles as pre-Lie skew braces, set-theoretic Hopf algebras $\&$ universal ${\cal R}$-matrices}
\author[Anastasia Doikou,  Bernard Rybo{\l}owicz, Paola Stefanelli]{Anastasia Doikou,  Bernard Rybo{\l}owicz and Paola Stefanelli}

\address[Anastasia Doikou $\&$ Bernard Rybo{\l}owicz] {Department of Mathematics, Heriot-Watt University,
Edinburgh EH14 4AS $\&$ Maxwell Institute for Mathematical Sciences, Edinburgh EH8 9BT, UK}
\email{a.doikou@hw.ac.uk,\ b.rybolowicz@hw.ac.uk}

\address[Paola Stefanelli] {Dipartimento di Matematica e Fisica “Ennio De Giorgi”,  Università 
del Salento, Via Provinciale Lecce-Arnesano, 73100 Lecce, Italy}
\email{paola.stefanelli@unisalento.it}

  \usepackage{latexsym} 
  \usepackage[all]{xy}
  \usepackage{amsfonts} 
  \usepackage{amsthm} 
  \usepackage{amsmath} 
  \usepackage{amssymb}
  \usepackage{pifont}  
  \usepackage{enumerate}
  \usepackage{dcolumn}
  \usepackage{comment}
  \usepackage{hyperref}  
  \usepackage{tikz-cd}
\usepackage[english]{babel}
  \usepackage{amsfonts} 
\usepackage[utf8]{inputenc}
\usepackage{graphicx}
\usepackage{subcaption}
\usepackage[colorinlistoftodos]{todonotes}
\usepackage{indentfirst}
\usepackage{pifont}  
  \usepackage{enumerate}
  \usepackage{dcolumn}
  \usepackage{comment}
\usepackage[capitalise]{cleveref}

 \usepackage{tikz}
\usetikzlibrary{arrows}

 \newcolumntype{2}{D{.}{}{2.0}}
  \xyoption{2cell}

\newcommand{\hiddenpower}[2] { \ifnum \numexpr#2=1 #1 \else #1^#2 \fi }
\numberwithin{equation}{section}

\def\be{\begin{equation}}
\def\ee{\end{equation}}
\def\ba{\begin{eqnarray}}
\def\ea{\end{eqnarray}}

\newcommand{\cal}{\mathcal}

\usepackage{xparse}
\ExplSyntaxOn
\newcounter{diff_order}
\newcounter{diff_power}

\newcommand{\rawdiff}[3]
{
	\setcounter{diff_order}{0}
	\clist_map_inline:nn{#3}{\stepcounter{diff_order}}
	
	\frac{\hiddenpower{#1}{\thediff_order} #2}
	{
		\def\old_var{DefaultValue}
		\setcounter{diff_power}{0}
		
		\clist_map_inline:nn{#3}
		{
			\def\new_var{##1}
			\ifnum \thediff_power=0
				\stepcounter{diff_power}
			\else
				\tl_if_eq:NNTF \new_var \old_var
				{\stepcounter{diff_power}}
				{
					#1 \hiddenpower{\old_var}{\thediff_power}
					\setcounter{diff_power}{1}
				}
			\fi

			\def\old_var{##1}
		}
		
		#1 \hiddenpower{\old_var}{\thediff_power}
	}
}
\def\Label#1{\label{#1}\ifmmode\llap{[#1] }\else 
  \marginpar{\smash{\hbox{\tiny [#1]}}}\fi} 
  \def\Label{\label} 

\ExplSyntaxOff

\newlength{\bibitemsep}
\newlength{\bibparskip}
\let\oldthebibliography\thebibliography
\renewcommand\thebibliography[1]{%
  \oldthebibliography{#1}%
  \setlength{\parskip}{\bibitemsep}%
  \setlength{\itemsep}{\bibparskip}%
}

\newtheorem{thm}{Theorem}[section]
\newtheorem{lemma}[thm]{Lemma}
\newtheorem{cor}[thm]{Corollary}
\newtheorem{pro}[thm]{Proposition}
\newtheorem{defn}[thm]{Definition}

\newtheorem{rem}[thm]{Remark}
\newtheorem{exa}[thm]{Example}

\newcommand{\Aut}{\operatorname{Aut}}
\newcommand{\Sym}{\operatorname{Sym}}

\newcommand{\id}{\operatorname{id}}

\newenvironment{wideregion}[1][0mm]
    {\vspace{#1}\list{}{\leftmargin-2cm\rightmargin\leftmargin}\item\relax}
    {\endlist}

\newenvironment{widegather }{\wideregion[-9mm]\gather}{\endgather\endwideregion}

\begin{document}
\vskip 0.8in

\hfill
 \begin{abstract} 
We present connections between left non-degenerate solutions of the set-theoretic braid equation and left shelves using Drinfel'd homomorphisms. We generalize the notion of affine quandle, by using  heap endomorphisms and metahomomorphisms, and identify the underlying Yang-Baxter algebra for solutions of the braid equation associated to a given quandle. We introduce the notion of the pre-Lie skew brace and identify certain affine quandles that give rise to pre-Lie skew braces. Generalisations of the braiding of a group,  associated to set-theoretic solutions of  the braid equation are also presented. These generalized structures encode part of the underlying Hopf algebra. We then introduce the quasi-triangular (quasi) Hopf algebras and universal ${\cal R}$-matrices for rack and set-theoretic algebras. Generic set-theoretic solutions coming from heap endomorphisms are also identified.

\end{abstract}
\maketitle

\date{}




\section{Introduction} 

\noindent The Yang-Baxter equation is a key mathematical object in the theory of quantum integrable
models and solvable statistical systems \cite{FRT}, as well as in the formulation of
quantum groups \cite{FRT, Jimbo1, Jimbo2, Drinfel'd}. It was first introduced in \cite{Yang} as the main mathematical tool for the investigation 
of quantum systems with many particle interactions, and in \cite{Baxter} for the study of a
two-dimensional solvable statistical model known as the anisotropic Heisenberg magnet or the XYZ model. In early 90s
Drinfeld \cite{Dr92} suggested the idea of
set-theoretic solutions to the Yang-Baxter equation and since then, set-theoretic solutions
have been extensively investigated primarily by means of representations of the braid group.
It is worth noting that set-theoretic solutions and Yang-Baxter 
maps have been also studied
within the context of classical discrete integrable systems connected also to the notion of
Darboux-B\"acklund transformation in the Lax pair formulation 
\cite{Adler, Papa, Veselov}. In
classical integrable systems usually a Poisson structure exists associated to a classical
$r$-matrix, which is a solution of the classical Yang-Baxter equation \cite{FT}. 
Various connections between the
set-theoretic Yang-Baxter equation and geometric crystals \cite{Crystals, Crystal2}, or soliton cellular automatons \cite{Cell, Cell2} have been also demonstrated.
The theory of braces was established around 2005, when Wolfgang Rump developed an algebraic structure that generalizes radical rings called a brace \cite{Ru05, Ru07} to describe all finite involutive set-theoretic solutions of the braid (and hence the Yang-Baxter) equation, whereas  skew-braces were developed by Guarnieri and
Vendramin to describe non-involutive solutions \cite{GuaVen}.

The investigation of set-theoretic solutions of the braid (and Yang-Baxter) equation and the associated algebraic structures is an emerging research area that has been particularly fruitful, and a significant number of related studies has been produced over the past
few years (see for instance  \cite{Bachi}--\cite{LebMan} \cite{GatMaj}--\cite{Gat2}, \cite{JesKub, Smo1, Smo2}).  In \cite{DoiSmo}--\cite{DoRy23} key links between set-theoretic solutions and quantum
integrable systems and the associated quantum algebras were uncovered. Moreover, interesting correspondences between Yang-Baxter algebras and pre-Lie algebras found in \cite{Doikoupre}, whereas connections between braces and
pre-Lie rings established in \cite{Smopre}.
In general, the theory of the  set-theoretic Yang-Baxter equation turns out to have intriguing interlinks with numerous research
areas, for example with group theory (Garside groups, regular subgroups, factorized
groups, see for example \cite{Bachi, JesKub, Jesp2}), algebraic number theory, Hopf-Galois extensions \cite{Bachi, Smo1}, non-commutative rings \cite{Ru07}, Knot theory \cite{Jo82}--\cite{LebVen}, Hopf algebras, quantum groups \cite{EtScSo99}, universal algebras, groupoids \cite{Pili}, semi-braces \cite{CaCoSt17}, trusses and heaps \cite{Brz:Lietruss}, pointed Hopf algebras, Yetter-Drinfield modules and Nichols algebras \cite{Andru}.

We now introduce the notion of the set-theoretic braid equation. Following \cite{Dr92}, given a set $X$, a map $r:X\times X\to X\times X$ is said 
to be a \emph{set-theoretic solution of the Yang-Baxter equation}, if $r$ satisfies the \emph{braid identity}
\begin{equation}\label{eq:braid}
\left(r\times\id_X\right)
\left(\id_X\times r\right)
\left(r\times\id_X\right)
= 
\left(\id_X\times r\right)
\left(r\times\id_X\right)
\left(\id_X\times r\right).
\end{equation}
Starting now, we call such a map $r$ simply a \emph{solution} and write $\left(X,r\right)$ to 
denote a solution $r$ on a set $X$. Besides, if we write $r\left(a,b\right) = \left(\sigma_a\left(b\right), 
\tau_b\left(a\right)\right)$, with $\sigma_a, \tau_a$ maps from $X$ into itself,  then $r$ is said to 
be \emph{left non-degenerate} if $\sigma_a\in \Sym_X$, for every $a\in X$, \emph{right non-degenerate} 
if $\tau_a\in \Sym_X$,  for every $a\in X$, and \emph{non-degenerate} if $r$ is both left and right non-degenerate.
Furthermore, if $r$ is a solution such that $r^2=\id_{X\times X}$, then $r$ is said to be \emph{involutive}. \\
It is also worth recalling the connection between the set-theoretic braid equation (\ref{eq:braid}) and the set-theoretic Yang-Baxter equation (see also \cite{Jimbo1, Jimbo2}). We  introduce the map $R: X\times X\rightarrow X\times X,$ such that $R =  r\pi,$ 
where $\pi: X\times X\rightarrow X\times X$ 
is the flip map: $\pi(x,y) = (y,x).$ Hence, $R(y,x) =( \sigma_x(y), \tau_y(x)),$ and it satisfies the set-theoretic Yang-Baxter equation:
\begin{equation}
R_{12}\ R_{13}\  R_{23} =R_{23}\ R_{13}\ R_{12},  \label{YBE}
\end{equation} 
where we denote $R_{12}(y,x,w) = (\sigma_x(y), \tau_y(x),w),$ $R_{23}(w,y,x) = (w, \sigma_x(y), \tau_y(x))$ 
and \\  $R_{13}(y,w,x) = (\sigma_x(y), w,\tau_y(x)).$

To study solutions, Etingof, Schedler, and Soloviev \cite{EtScSo99} introduced 
the notion of equivalence, or isomorphism, of solutions. Specifically, following 
\cite{CeJeOk14}, two solutions $\left(X,r\right)$ and $\left(Y,s\right)$ 
are said to be \emph{homomorphic} if there exist a map $f:X\to Y,$ such 
that $\left(f\times f\right)r = s\left(f\times f\right)$. In particular, $r$ and $s$ 
are said to be \emph{homomorphic via $f$}. If in addition $f$ is bijective, 
then $\left(X,r\right)$ and $\left(Y,s\right)$ are said to be \emph{equivalent} or \emph{isomorphic}.

To date, left non-degenerate solutions have been investigated through different algebraic structures such as cycle sets \cite{Ru05}, q-cycle sets \cite{Ru19, CaCaSt22}, twisted Ward left quasigroups \cite{StVo21}, semi-braces \cite{CaCoSt17}, semitrusses \cite{CoJeVaVe22},  and left simple semigroups \cite{CoJeKuVaVe23}.  In any case, given a left non-degenerate solution $\left(X,r\right)$, one can define another left
non-degenerate solution $\left(X,r'\right)$ on the same underlying set $X$ 
connected to $r$, namely, the map $r':X\times X\to X\times  X$ defined by 
$r'\left(a,b\right) = \left(b,\, \sigma_b\tau_{\sigma^{-1}_a\left(b\right)}\left(a\right)\right)$, 
for all $a,b\in X$. Such a solution is called \emph{derived solution of $r$}. Set $a\triangleright 
b:= \sigma_b\tau_{\sigma^{-1}_a\left(b\right)}\left(a\right)$, for all $a,b\in X$, 
we obtain that $\left(X, \triangleright\right)$ is a left shelf.  Conversely, if a set $X$ is endowed with a binary 
operation such that $\left(X, \triangleright\right)$ is a left shelf, then the map $r_\triangleright:X\times X\to X\times  X$ 
defined by $r_\triangleright\left(a,b\right) = \left(b,\, \sigma_b\tau_{\sigma^{-1}_a\left(b\right)}\left(a\right)\right)$, 
for all $a,b\in X$, is a left non-degenerate solution, known as the derived solution.

In this paper, we show that the well-known correspondence between derived solutions and left shelves can be 
``improved'' describing all left non-degenerate solutions. To obtain this description, the notion 
of D-isomorphic (Drinfel'd isomorphic) solutions is crucial. 
In particular, we show that $D$-isomorphic classes of left non-degenerate solutions correspond 
to isomorphic classes of left shelves endowed with special endomorphisms (these left shelves 
are exactly the same ones associated with their derived).

The key outcomes of this investigation are essentially summarized in Theorem \ref{le:lndsol}, where it is shown that every left non-degenerate solution can be obtained from a shelf solution by finding special shelf automorphisms, and in Sections 3-5.
In Sections 3-5 we  generalize the notion of affine quandle,  by using  heap endomorphisms 
and metahomomorphisms, and identify the Yang-Baxter algebra for solutions of the braid equation associated to a given rack/quandle. Motivated by the notion of pre-Lie algebras (also studied under the name chronological algebras) \cite{Chrono, PreLie, Pre1} (see also \cite{Bai, Manchon} for a recent  reviews) we introduce a novel algebraic structure called  {\it pre-Lie skew brace} to describe the underlying algebraic structures associated to certain set-theoretic solutions of the braid equation.  In fact, we identify families of affine quandles that produce pre-Lie skew braces. We then consider the linearized version of the braid equation and we introduce the notion of the quasi-triangular Hopf algebras associated to set-theoretic solutions; the rack type
universal ${\cal R}-$matrix is also derived.  By identifying a suitable admissible Drinfel'd twist \cite{Drinfel'd, Drinfel'd2} we are able to extract the general set-theoretic universal ${\cal R}-$matrix. Our findings on universal admissible Drinfel'd twists are consistent with Theorem \ref{le:lndsol}. 

More precisely, the outline of the paper is as follows.
We start with \cref{Sec:shelf} which is divided into two subsections. The first subsection contains preliminaries on shelves. In the second subsection, in \cref{def:Drinfiso}, we introduce the notion of a Drinfel'd homomorphism between set-theoretic solutions of the Yang-Baxter equation, based on the well known idea of Drinfel'd twists. When one considers a Drinfel'd automorphism, one gets a set-theoretic analogue of admissible twists of quasitriangular Hopf-algebras. We show that every left non-degenerate solution is Drinfel'd isomorphic to a solution given by a shelf, see \cref{Th.r->r'}. In fact, by \cref{le:lndsol}, every left non-degenerate solution can be acquired from a shelf solution by finding special automorphisms of the shelf itself, which we name  \emph{twists}, see \cref{def:twist:shelf}. In particular, bijective solutions correspond to racks. We show that two homomorphic shelves with twists have D-homomorphic solutions, and bijectivity of a homomorphism of shelves implies D-isomorphism of solutions. We conclude \cref{Sec:shelf} with examples of solutions given by shelves and twists on them.

\cref{Sec:Quandle&Lie} is divided into three subsections. In the first subsection, we present a definition of a Yang-Baxter algebra $(X,m)$ associated with a given  solution $(X,r)$ as a binary operation $m:X\times X\to X$ such that $mr=m,$ i.e., by the Yang-Baxter algebra in \cref{qua}; specifically, we understand an algebra in the sense of a  universal algebra that is invariant under the entanglement of a particular set-theoretic solution of the Yang-Baxter equation. We show that for any set endowed with a left non-degenerate solution, the Yang-Baxter algebra exists, see \cref{lemmagen-2}. We conclude the first subsection with examples of Yang-Baxter algebras for solutions given by skew braces and near braces. In the second subsection, we introduce three generalisations of affine quandles to the non-abelian case and study their Yang-Baxter algebras, see examples \ref{ex:op:1},\ref{ex:op:2} and \ref{ex:op:3}. 
In particular, endofunctions called \emph{metahomomorphisms} and \emph{heap endomorphisms} allow us to define such generalized quandles.
Moreover, if the underlying group of those quandles is abelian and metahomomorphisms are group endomorphisms, all the examples collapse to the definition of an affine quandle. In addition, every heap endomorphism gives rise to quandle given by a metahomomorphism, see  \cref{lem:3stories}; nevertheless, the Yang-Baxter algebra operation obtained by a heap endomorphism is not necessarily given by a metahomomorphism, see \cref{rem:2metend} and \cref{prop1}. In the third and last subsection, we introduce a notion of the \emph{right pre-Lie skew brace}, see \cref{def1:prelie}.  In particular, every pre-Lie ring is a right Pre-Lie brace. Next, we show that Yang-Baxter algebras from examples \ref{ex:op:2}
and \ref{ex:op:3} admit a pre-Lie skew brace structure, see \cref{thm:right-pre-Lie}. If an underlying group of a pre-Lie skew brace is abelian, one can acquire a pre-Lie ring. We conclude the subsection with \cref{ex:trivial} in which we show that Lie rings associated with a Yang-Baxter algebra given by affine quandle have a zero Lie bracket.

In \cref{sec:4}, we present a generalized version of the braiding of a group, with some examples that have already appeared in \cite{DoRy23, DoiRyb22}, see \cref{def0}. We generalize it beyond the group structure to be compatible with Yang-Baxter algebras. We conclude the section with examples of such braidings, see examples \ref{ex:4:1},\ref{ex:4:2},\ref{ex:4:3} and \ref{ex:4:4}. This section serves as a precursor of the last section given  that the generalized braided structures encapsulate part of the associated underlying Hopf algebras, which is introduced next.

In \cref{sec:5},  which is divided in two subsections, we shift our focus from set-theoretic solutions to linearized versions and Hopf algebras. Specifically,  in the first subsection, introduce the Yang-Baxter algebras of rack/quandle type and set-theoretic solutions 
as quasitriangular (quasi)-Hopf algebras \cite{Drinfel'd, Drinfel'd2}.  Relevant earlier works in  \cite{Doikou1, DoGhVl, DoiRyb22} also in \cite{EtScSo99} and  \cite{Andru} in connection with pointed Hopf algebras,  and the bialgebras associated to racks are studied in \cite{LebMan, Lebed, braceh}.
We first introduce the quandle Hopf algebra and we then systematically construct the associated universal 
${\cal R}$-matrix.
In the second subsection, we suitably extend the quandle algebra and present the set-theoretic Yang-Baxter algebra, which is also a Hopf algebra. By means of a a suitable admissible universal Drinfel'd twist 
we construct the universal-set theoretic ${\cal R}-$matrix. We note that the linearized version of the Yang-Baxter equation allows the investigation of the full Hopf algebras and not only the study of the subset consisting of group-like elements that naturally appear in the set-theoretic frame.  Fundamental representations of the aforementioned  
algebras are also considered leading to the rack and general set-theoretic solutions of the Yang-Baxter equation. Generalized novel set-theoretic solutions coming from heap endomorphisms  and being compatible with the universal set-theoretic Yang-Baxter algebra are also introduced.

\section{From left shelves to set-theoretic solutions and vice versa}\label{Sec:shelf}

\noindent In this section, we study and summarize connections between left non-degenerate solutions of the set-theoretic braid equation and left shelves. The first subsection contains preliminaries on shelves. In the second subsection, we collect various results and remarks scattered in literature on shelves and set-theoretic braid equations. Motivated by Drinfel'd's theory of admissible twists of Hopf-algebras, we introduce a set-theoretic analogue of those twists, which we call Drinfel'd isomorphisms. Inspired by recent developments in linearized and set-theoretic settings, we present how starting with a shelf, we can construct all left non-degenerate set-theoretic solutions using a particular family of shelf isomorphisms, which we name twists.

\subsection{Left shelves}
This section contains preliminaries on left shelves.  For the first systematic study of shelves, we refer reader to \cite{Shelf-history2}.

\begin{defn}
    Let $X$ be a non-empty set and\, $\triangleright$\, a binary operation on $X$. Then, the pair $\left(X,\,\triangleright\right)$ 
is said to be a \emph{left shelf} if\, $\triangleright$\, is left self-distributive, namely, the identity
    \begin{equation}\label{eq:shelf}
        a\triangleright\left(b\triangleright c\right)
        = \left(a\triangleright b\right)\triangleright\left(a\triangleright c\right) 
    \end{equation}
    is satisfied, for all $a,b,c\in X$. Moreover, a left shelf $\left(X,\,\triangleright\right)$ is called 
    \begin{enumerate}
        \item  a \emph{left spindle} if $a\triangleright a = a$, for all $a\in X$;
        \item  a \emph{left rack} if $\left(X,\,\triangleright\right)$ is a \emph{left quasigroup}, i.e., the maps $L_a:X\to X$ defined by $L_a\left(b\right):= a\triangleright b$, 
        for all $b\in X$, are bijective, for every $a\in X$.
        \item a \emph{quandle} if $\left(X,\,\triangleright\right)$ is both a left spindle and a left rack.
    \end{enumerate}
\end{defn}

\begin{exa} \label{ex:trivial-sh}\hspace{1mm}
\begin{enumerate}
    \item 
    Let $X:=\mathbb{Z}_6$ the set of integers modulo $6$. Then
    	the binary operation of $X$ given by $a\triangleright b:= 2a + 2b$, for all $a,b\in X$, endows $X$ of a structure of a left shelf that is not a spindle.
    \item Given a set $X$ and considered the binary operation of $X$ given by $a\triangleright b:= a$, for all $a,b\in X$, 
we trivially obtain a left spindle (in particular,  the structure $\left(X,\,\triangleright\right)$ is a left-zero band).
    \item 
    Let $X:=\mathbb{Z}_4$ the set of integers modulo $4$. Then
    the binary operation of $X$ given by $a\triangleright b:= 2a + b$, for all $a,b\in X$, endows $X$ of a structure of a left rack that is not a quandle.
    \item Given a set $X$ and considered the binary operation of $X$ given by $a\triangleright b:= b$, for all $a,b\in X$, 
we trivially obtain a left quandle (in particular, the structure $\left(X,\,\triangleright\right)$  is a right-zero band).
\end{enumerate}
\end{exa}

\begin{defn}
    If $\left(X,\,\triangleright\right)$ and 
$\left(Y,\,\blacktriangleright\right)$ 
are left shelves, 
    a map $f:X\to Y$ is said to be a \emph{shelf homomorphism} if
    $f\left(a\triangleright b\right) = f\left(a\right)\blacktriangleright f\left(b\right)$, 
for all $a,b\in X$, or, equivalently, $fL_a\left(b\right) = \bar{L}_{f\left(a\right)}f\left(b\right)$, 
for all $a,b\in X$, where $\bar{L}_u\left(v\right) = u\blacktriangleright v$, for all $u,v\in Y$.
We denote the set of all endomorphisms of 
$\left(X,\,\triangleright\right)$\, by  
$\mathrm{End}(X,\,\triangleright)$ and the set of all automorphisms of 
$\left(X,\,\triangleright\right)$\, by  
$\mathrm{Aut}(X,\,\triangleright)$.
\end{defn}

\begin{rem}\label{rem:prop-L_a}
    By \eqref{eq:shelf}, one trivially obtains that $L_a\in End\left(X,\,\triangleright\right)$, for every $a\in X$. 
Moreover, the identity
    \begin{equation}\label{prop2-La}
        L_aL_b = L_{L_a\left(b\right)}L_a
    \end{equation}
    is satisfied, for all $a,b\in X$.
\end{rem}

The following proposition describes a well-known connection between 
left non-degenerate solutions and left shelves \cite{Sol, LebVen, Doikou1, DoGhVl, DoiRyb22} 
\begin{pro}\label{prop-ls<->lnds}
   If  $\left(X,\,\triangleright\right)$ is a left shelf, then the map $r_{\triangleright}:X\times X\to X\times X$ 
defined by $r_{\triangleright}\left(a,b\right) = \left(b,\,b\triangleright a\right)$, for all $a,b\in X$, is a left non-degenerate solution of derived type.\\ 
   Conversely, if $\left(X,r\right)$ is an arbitrary  left non-degenerate solution, then the structure 
$\left(X,\,\triangleright_r\right)$ is a left shelf where\, $\triangleright_r$\, is 
the binary operation on $X$ given by $a\triangleright_r b:= \sigma_a\tau_{\sigma^{-1}_b\left(a\right)}\left(b\right)$, for all $a,b\in X$. 
\end{pro}

\begin{rem}
   With reference to \cref{prop-ls<->lnds}, let us observe that:
 If $\left(X,\,\triangleright\right)$ is a left spindle, then the solution $r_\triangleright$ is 
square free. Moreover, if $\left(X,\,\triangleright\right)$ is a left rack, then the solution $r_{\triangleright}$ is non-degenerate.
\end{rem}

\begin{defn}
From now on, given a left shelf $\left(X,\,\triangleright\right)$, we call $r_{\triangleright}$ the \emph{solution associated to $\left(X,\,\triangleright\right)$}. Moreover, given a left 
non-degenerate solution $r$, we call $\left(X,\, \triangleright_r\right)$ the \emph{left shelf associated to $r$}.
\end{defn}

\begin{exa}\hspace{1mm}
\begin{enumerate}
    \item The solution associated to the left shelf $\left(X,\triangleright\right)$ in \cref{ex:trivial-sh}(2) is the map $r_{\triangleright}:X\times X\to X\times X$ given by $r_{\triangleright}\left(a,b\right) = \left(b,b\right)$ 
and it is clearly left non-degenerate and idempotent.
    \item The solution associated to the left rack $\left(X,\triangleright\right)$ in \cref{ex:trivial-sh}(4) 
is the map $r_{\triangleright}:X\times X\to X\times X$ given by 
$r_{\triangleright}\left(a,b\right) = \left(b,a\right)$, so it is the \emph{twist map} that is clearly non-degenerate and involutive.
\end{enumerate}
\end{exa}

\subsection{Left-non degenerate solutions and left shelves}

\noindent We  first  introduce the notion of a {\it Drinfel'd isomorphism} and recall some known results 
 \cite{Sol, LebVen, Doikou1, DoGhVl, DoiRyb22, CoJeVaVe22}.

\begin{defn}\label{def:Drinfiso}
Let $(X,r)$ and $(Y,s)$ be solutions. Then we say that a map $\varphi: X\times X\to Y\times Y$ 
is a \emph{Drinfel'd homomorphism} or in short \emph{D-homomorphism} if 
$$
\varphi\, r = s\, \varphi. 
$$
If $\varphi$ is a bijection, we call $\varphi$ a {\it $D$-isomorphism} and we say that $(X,r)$ 
and $(Y, s)$ are {\it $D$-isomorphic (via $\varphi$)}, and we denote it by $r\cong_D s$.
\end{defn}

\begin{lemma}
If $(X,r)$ and $(Y,s)$ are homomorphic (resp. equivalent) solutions via $f:X\to Y$, then they 
are $D$-homomorphic (resp. $D$-isomomorphic) via $\varphi = f\times f$.
\end{lemma}

\noindent The converse of the previous fact does not hold.

\begin{exa}
Let us consider a set $U(\mathbb{Z}/8\mathbb{Z}):=\{1,3,5,7\}$ and the following solution for all $a,b\in U(\mathbb{Z}/8\mathbb{Z})$
$$
r(a,b)=(1-a+ab,(1-a+ab)^{-1}ab),
$$
where $+$ and juxtaposition are addition and multiplication of integers modulo $8.$ 
Then one can show that $r$ is not equivalent to the flip map, but it is $D$-isomorphic with the flip map.
\end{exa}

\begin{lemma}\label{Th.r->r'}
    Let $(X,r)$ be a left non-degenerate solution and $(X,r')$ be the derived solution of $(X,r)$. Then $r$ is  $D$-isomorphic to $r'$.
\end{lemma}

\begin{proof}
Let $\varphi:X\times X\to X\times X$ be the map defined by
$$
\varphi(a,b):=(a,\sigma_a(b)),
$$ 
for all $a,b\in X$. 
Since $r$ is left non-degenerate, $\varphi$ is bijective and $\varphi^{-1}\left(a,b\right) = 
\left(a,\sigma^{-1}_a\left(b\right)\right)$, for all $a,b\in X$. Then 
$$\varphi r \varphi^{-1}(a,b)=\varphi r(a,\sigma_a^{-1}(b))=\varphi(\sigma_a
\sigma_a^{-1}(b),\tau_{\sigma^{-1}_a(b)}(a))=(b, \sigma_b(\tau_{\sigma^{-1}_a(b)}(a))).$$
That is $r\cong_D r'$.
\end{proof}

\begin{lemma}
Given a set $X$, a map $r:X\times X\to X\times X,$ such that $r(a,b) =(\sigma_a(b), \tau_b(a)),$ for all $a, b \in X,$ satisfies \eqref{eq:braid} if and only if the following 
three conditions are satisfied for all $a,b,c\in X$
\begin{equation}
   \sigma_a\sigma_b = \sigma_{\sigma_a\left(b\right)}\sigma_{\tau_b\left(a\right)}\label{eq:braidI}\tag{bI}
\end{equation}
\begin{equation}
    \sigma_{\tau_{\sigma_b\left(c\right)}\left(a\right)}
    \tau_c\left(b\right) 
   = \tau_{\sigma_{\tau_b\left(a\right)}\left(c\right)}\sigma_a\left(b\right) \label{eq:braidII}\tag{bII}
\end{equation}
\begin{equation}
\tau_c\tau_b = \tau_{\tau_c\left(b\right)}\tau_{\sigma_b\left(c\right)}.\label{eq:braidIII}\tag{bIII}
\end{equation}
\end{lemma}

\begin{defn}\label{def:twist:shelf}
    Let $(X,\triangleright)$ be a left shelf. We say that $\varphi:X\to \mathrm{Aut(X,\triangleright)},$ $a\mapsto \varphi_a$ is a \emph{twist} if for all $a,b\in X$, 
     \begin{equation}\label{eq:varphi}
        \varphi_a\varphi_b = \varphi_{\varphi_a\left(b\right)}\varphi_{\varphi^{-1}_{\varphi_a\left(b\right)}L_{\varphi_a\left(b\right)}\left(a\right)}.
    \end{equation}
\end{defn}

In the following, we show that any left non-degenerate solution  $\left(X,\, r\right)$ can be described in terms of the left shelf $\left(X,\,\triangleright_r\right)$ defined in \cref{prop-ls<->lnds} 
and its twist.

\begin{thm}\label{le:lndsol}
    Let $\left(X,\,\triangleright\right)$ be a left shelf and $\varphi:X
\to \mathrm{Sym}_X,$ $a\mapsto \varphi_a$. Then, the function 
$r_\varphi:X\times X\to X\times X$ defined by
    \begin{equation}\label{prop:solform}
        r_\varphi\left(a, b\right)  
        = \left(\varphi_a\left(b\right),\, \varphi^{-1}_{\varphi_a\left(b\right)}\left(\varphi_a\left(b\right)\triangleright a\right)\right),  
    \end{equation}
    for all $a,b\in X$, is a solution if and only if $\varphi$ is a twist. Moreover, 
    any left non-degenerate solution can be obtained that way.
\end{thm}
\begin{proof}
    First, let us assume that $\varphi$ is a twist. Then \eqref{eq:braidI} follows by \eqref{eq:varphi}. Set $\tau_{b}\left(a\right):= \varphi^{-1}_{\varphi_a\left(b\right)}L_{\varphi_a\left(b\right)}\left(a\right)$, for all $a,b\in X$. If $a,b,c\in X$, we get that 
     $$
    \begin{aligned}
     \tau_{\varphi_{\tau_b\left(a\right)}\left(c\right)}\varphi_a\left(b\right) 
    &= \varphi^{-1}_{\varphi_{\varphi_a\left(b\right)}\varphi_{\tau_b\left(a\right)}\left(c\right)}
    L_{\varphi_{\varphi_a\left(b\right)}\varphi_{\tau_b\left(a\right)}\left(c\right)}\varphi_a\left(b\right)\\
    &=\varphi^{-1}_{\varphi_a\varphi_b\left(c\right)}L_{\varphi_a\varphi_b\left(c\right)}{\varphi_a\left(b\right)}&\mbox{by \eqref{eq:braidI}}\\
    &= \varphi^{-1}_{\varphi_a\varphi_b\left(c\right)}\varphi_aL_{\varphi_b\left(c\right)}\left(b\right)
    &\mbox{$\varphi_a\in End\left(X,\,\triangleright\right)$}\\
    &=\varphi_{\tau_{\varphi_b\left(c\right)}\left(a\right)}\varphi^{-1}_{\varphi_b\left(c\right)}L_{\varphi_b\left(c\right)
    }\left(b\right)&\mbox{by \eqref{eq:braidI}}\\
    &= \varphi_{\tau_{\varphi_b\left(c\right)}\left(a\right)}\tau_c\left(b\right),
    \end{aligned}
    $$
hence \eqref{eq:braidII} holds.

In particular, from the previous equalities,  we obtain that
\begin{equation}\label{eq:2}
    \varphi_{\tau_{\varphi_b\left(c\right)}\left(a\right)}\tau_c\left(b\right) =
 \varphi^{-1}_{\varphi_a\varphi_b\left(c\right)}\varphi_aL_{\varphi_b\left(c\right)}\left(b\right).
\end{equation}
Moreover, by applying twice \eqref{eq:braidI}, 
\begin{equation}\label{eq:2.2}
    \varphi_{\varphi_a\varphi_b\left(c\right)}\varphi_{\tau_{\varphi_{\tau_b\left(a\right)}\left(c\right)}\varphi_a\left(b\right)}
    = \varphi_{\varphi_{\varphi_a\left(b\right)}\varphi_{\tau_{b}\left(a\right)}\left(c\right)}
\varphi_{\tau_{\varphi_{\tau_b\left(a\right)}\left(c\right)}\varphi_a\left(b\right)}
    = \varphi_{\varphi_a\left(b\right)}\varphi_{\varphi_{\tau_b\left(a\right)}\left(c\right)}.
\end{equation}
Thus,
$$
\begin{aligned}
\tau_{\tau_c\left(b\right)}\tau_{\varphi_b\left(c\right)}\left(a\right)&=
 \varphi^{-1}_{\varphi_{\tau_{\varphi_b\left(c\right)}\left(a\right)}\tau_c\left(b\right)}
L_{\varphi_{\tau_{\varphi_b\left(c\right)}\left(a\right)}\tau_c\left(b\right)}\tau_{\varphi_b\left(c\right)}\left(a\right)
\\
&= \varphi^{-1}_{\tau_{\varphi_{\tau_b\left(a\right)}\left(c\right)}\varphi_a\left(b\right)}
L_{\varphi^{-1}_{\varphi_a\varphi_b\left(c\right)}\varphi_aL_{\varphi_b\left(c\right)}\left(b\right)}
\varphi^{-1}_{\varphi_a\varphi_b\left(c\right)}
L_{\varphi_a\varphi_b\left(c\right)}\left(a\right)
&\mbox{by \eqref{eq:braidII} and \eqref{eq:2}}\\
&= \varphi^{-1}_{\tau_{\varphi_{\tau_b\left(a\right)}\left(c\right)}\varphi_a\left(b\right)}\varphi^{-1}_{\varphi_a\varphi_b\left(c\right)}
L_{L_{\varphi_a\varphi_b\left(c\right)}\varphi_a\left(b\right)}L_{\varphi_a\varphi_b\left(c\right)}\left(a\right)
&\mbox{$\varphi^{-1}_{\varphi_a\varphi_b\left(c\right)},\varphi_a\in End\left(X,\,\triangleright\right)$}\\
&= \varphi^{-1}_{\varphi_{\tau_b\left(a\right)}\left(c\right)}
\varphi^{-1}_{\varphi_a\left(b\right)}
L_{\varphi_a\varphi_b\left(c\right)}L_{\varphi_a\left(b\right)}\left(a\right).
&\mbox{by \eqref{eq:2.2} and \eqref{prop2-La}}\\
&= \varphi^{-1}_{\varphi_{\tau_b\left(a\right)}\left(c\right)}
    L_{\varphi^{-1}_{\varphi_a\left(b\right)}\varphi_a\varphi_b\left(c\right)}
    \varphi^{-1}_{\varphi_a\left(b\right)}L_{\varphi_a\left(b\right)}\left(a\right)
    &\mbox{$\varphi_a\in End\left(X,\,\triangleright\right)$}\\
&= \varphi^{-1}_{\varphi_{\tau_b\left(a\right)}\left(c\right)}
    L_{\varphi_{\tau_b\left(a\right)}\left(c\right)}\tau_b\left(a\right)
    &\mbox{by \eqref{eq:braidI}}\\ 
    &=\tau_{c}\tau_{b}\left(a\right),
    \end{aligned}
$$
and we obtain that \eqref{eq:braidIII} is satisfied.\\
Conversely, if the map $r_\varphi$ is a solution on the set $X$. Then, \eqref{eq:braidI} 
coincides with the identity \eqref{eq:varphi}. Furthermore, by applying \eqref{eq:braidI} to the identity 
\eqref{eq:braidII} and by  looking at the previous computations, we get that 
$\varphi^{-1}_{\varphi_a\varphi_b\left(c\right)}\varphi_a  L_{\varphi_b\left(c\right)}\left(b\right) = 
\varphi^{-1}_{\varphi_a\varphi_b\left(c\right)}L_{\varphi_a\varphi_b\left(c\right)}\varphi_a\left(b\right)$, for all $a,b,c\in X$, i.e., 
$\varphi_a  L_{\varphi_b\left(c\right)}\left(b\right) = L_{\varphi_a\varphi_b\left(c\right)}\varphi_a\left(b\right)$, 
for all $a,b,c\in X$. Equivalently, 
$\varphi_{a}L_b\left(c\right) = L_{\varphi_{a}\left(b\right)}\varphi_{a}\left(c\right)$, for all $a,b,c\in X$, 
i.e., $\varphi_{a}\in End\left(X,\, \triangleright\right)$, for every $a\in X$.
Therefore, the first claim follows.

To observe that any left non-degenerate solution $r(a,b):= \left(\sigma_a\left(b\right), \tau_b\left(a\right)\right)$ 
can be obtained that way, it is enough to show that $\sigma$ is a twist of the shelf $(X,\triangleright_r),$ where $a\triangleright_r b:= \sigma_a\tau_{\sigma^{-1}_b\left(a\right)}\left(b\right).$ Indeed, for all $a,b,c\in X$, 
    $$
    \begin{aligned}
       \sigma_a\left(b\triangleright_r c\right)
        &= \sigma_{\sigma_a\left(b\right)}
        \sigma_{\tau_b\left(a\right)}\tau_{\sigma^{-1}_c\left(b\right)}\left(c\right) &\mbox{by \eqref{eq:braidI}}\\
        &= \sigma_{\sigma_a\left(b\right)}\sigma_{\tau_{\sigma_c\left(\sigma_c^{-1}\left(b\right)\right)}\left(a\right)}\
\tau_{\sigma^{-1}_c\left(b\right)}\left(c\right)\\
        &= \sigma_{\sigma_a\left(b\right)}\tau_{\sigma_{\tau_c\left(a\right)}\sigma^{-1}_c\left(b\right)}\sigma_a
\left(c\right)&\mbox{by \eqref{eq:braidII}}\\
        &= \sigma_{\sigma_a\left(b\right)}
        \tau_{\sigma^{-1}_{\sigma_a\left(c\right)}\sigma_a\left(b\right)}\sigma_a\left(c\right),&\mbox{by \eqref{eq:braidI}}\\
        &= \sigma_a\left(b\right)\triangleright_r \sigma_a\left(c\right),
    \end{aligned}
    $$
and $\sigma_a$ is a shelf homomorphism. Clearly, $\tau_b\left(a\right) = \sigma^{-1}_{\sigma_a\left(b\right)}\left(\sigma_a\left(b\right)\triangleright a\right)$, and the map $r$ can be written 
as in \eqref{prop:solform}. One can easily show that the identity 
\eqref{eq:varphi} holds. Therefore, the statement is proven.
\end{proof}

\begin{cor}\label{cor:bij-lnd}
A left non-degenerate solution $(X,r)$ is bijective if and only if $(X,\triangleright_r)$ is a rack.
\end{cor}
\begin{proof}
 This follows from the fact that $r$ is invertible if and only if $r_{\triangleright_r}$ is invertible. 
\end{proof}

As a consequence of \cref{cor:bij-lnd} and \cite[Corollary 6]{CaCaSt22} we can state the following result.
\begin{cor}
If $X$ is a finite set, a left non-degenerate solution $\left(X,r\right)$ is right non-degenerate and bijective if and only if $(X,\triangleright_r)$ is a rack.
\end{cor}

\begin{rem}
More generally, we obviously get that a bijective left non-degenerate solution $\left(X,\,r\right)$ is non-degenerate if and only if $\left(X,\,\triangleright_r\right)$ is a rack such that
    \begin{align*}
    \forall\ b,c\in X\quad \exists a\in X\quad \varphi_{\varphi_a\left(b\right)}\left(c\right)
    = L_{\varphi_a\left(b\right)}\left(a\right).
    \end{align*}
\end{rem}

\begin{rem}
If $\left(X,\,\triangleright\right)$ is a left shelf, $\varphi_a:X\to X$ a bijective map, for every $a\in X$, set $a\cdot b:= \varphi^{-1}_a\left(b\right)$ and $a:b:= \varphi^{-1}_a\left(a\triangleright b\right)$, for all $a,b\in X$, one can check that the map $r$ in \eqref{prop:solform} is a solution 
if and only if 
   $\left(X,\,\cdot,\,:\right)$ is a \emph{q-cycle set}, namely the structure introduced by Rump in \cite{Ru19}.   
   In this way, we obtain a connection between \cref{le:lndsol} and the result provided by Rump in \cite[Proposition 1]{Ru19} on the correspondence between left non-degenerate  solutions on a set $X$ and q-cycle sets on the same set $X$ (see also \cite[p. 2]{CaCaSt22}).
\end{rem}

\begin{cor}\label{cor:ide}
    All the left non-degenerate idempotent solutions on a set $X$ are the maps $r:X\times X\to X\times X$ defined by
    \begin{align*}
        r\left(a, b\right)  
        = \left(\varphi_a\left(b\right), \varphi^{-1}_{\varphi_a\left(b\right)}\varphi_a\left(b\right)\right)
    \end{align*}
    satisfying \eqref{eq:braidI}, where $\varphi_a:X\to X$ is  a bijective map, for every $a\in X$.
 
\end{cor}
\begin{proof}
Let $\left(X,\triangleright\right)$ be the trivial left shelf in \cref{ex:trivial-sh}(2) where $a\triangleright b = a$, for all $a,b\in X$. 
Since every bijection on $\left(X,\triangleright\right)$ is obviously an automorphism of the same structure, the claim follows by \cref{le:lndsol}.
\end{proof}

\begin{rem}\label{rem:lnd-ide}
Set $a\ast b:= \varphi^{-1}_a\left(b\right)$, for all $a,b\in X$, it is easy to check that the map $r$ in \cref{cor:ide} satisfies \eqref{eq:braidI} if and only if $\left(a\ast b\right)\ast \left(a\ast c\right) = \left(b\ast b\right)\ast \left(b\ast c\right)$, for all $a,b,c\in X$, i.e, the structure $\left(X,\ast\right)$ is a \emph{twisted Ward left quasigroup}. Hence, \cref{cor:ide} coincides with the result provided by Stanovsk\'{y} and  Vojt\v{e}chovsk\'{y} in \cite[Proposition 2.2]{StVo21} in which they establish a correspondence between left non-degenerate idempotent solutions on a set $X$ and twisted Ward left quasigroups on the same set $X$.
\end{rem}

\begin{cor}\label{cor:inv}
    All and only the left non-degenerate involutive solutions on a set $X$ are the maps $r:X\times X\to X\times X$ defined by
    \begin{align*}
        r\left(a, b\right)  
        = \left(\varphi_a\left(b\right), \varphi^{-1}_{\varphi_a\left(b\right)}\left(a\right)\right)
    \end{align*}
    satisfying \eqref{eq:braidI} where $\varphi_a:X\to X$ is a bijective map, for every $a\in X$.
\end{cor}
\begin{proof}
  Let $\left(X,\,\triangleright\right)$ be the trivial left rack in \cref{ex:trivial-sh}(4) where $a\triangleright b = b$, for all $a,b\in X$. 
  Since every bijection on $\left(X,\,\triangleright\right)$ is obviously an automorphism of the same structure, the claim follows by \cref{le:lndsol}.
\end{proof}

\begin{rem}\label{rem:lnd-inv}
    Set $a\cdot b:= \varphi^{-1}_a\left(b\right)$, for all $a,b\in X$, it is easy to check that the map $r$ satisfies \eqref{eq:braidI} if and only if
    $\left(a\cdot b\right)\cdot \left(a\cdot c\right) = \left(b\cdot a\right)\cdot \left(b\cdot c\right)$, for all $a,b,c\in X$, i.e., the structure $\left(X,\,\cdot\right)$ is a \emph{cycle set}. Thus, \cref{cor:inv} coincides with the result provided by Rump in \cite{Ru05} on the correspondence 
between left non-degenerate involutive solutions on a set $X$ and cycle sets on the same set $X$.
\end{rem}

\begin{rem}
    The description of left non-degenerate solutions in terms of automorphisms of shelves in \cref{le:lndsol} 
allows for explicitly obtaining the Drinfel'd twists that realize the conjugation between a left-non degenerate solution and its derived. 
    We note that a similar connection between derived solutions/shelves and left non-degenerate solutions also appears 
in \cite[p. 145]{Ru19}, even if not explicitly made.
\end{rem}

\begin{lemma}\label{th:T<->L}
Let $(X,\triangleright),$ $(Y,\blacktriangleright)$ be shelves, $f:(X,\triangleright)\to (Y,\blacktriangleright)$ 
be a homomorphism of shelves and $\varphi:X\times X\to X$, $\psi:Y\times Y\to Y$ be twists of the shelves. Then 
$$\Phi\left(a,\,b\right)= \left(f\left(a\right),\, \psi^{-1}_{f\left(a\right)}f\varphi_a\left(b\right)\right)
$$
for all $a,b\in X,$ is a $D$-homomorphism of solutions $(X,r_\phi)$ and $(Y,r_{\psi})$. Moreover, $f$ is bijective if and only if $\Phi$ is a $D$-isomorphism.
\end{lemma}

\begin{proof} 
Observe that for any $a,b\in X$, we have that
$$
\begin{aligned}
    \Phi r_\varphi\left(a,b\right)
    &= \left(f\varphi_a\left(b\right),\, \psi^{-1}_{f\varphi_a\left(b\right)}f\varphi_{\varphi_a\left(b\right)}
\varphi^{-1}_{\varphi_a\left(b\right)}L_{\varphi_a\left(b\right)}\left(a\right)\right)
    = \left(f\varphi_a\left(b\right),\, \psi^{-1}_{f\varphi_a\left(b\right)}\bar{L}_{f\varphi_a\left(b\right)}f\left(a\right)\right)\\
    &= \left(\psi_{f\left(a\right)}\psi^{-1}_{f\left(a\right)} f\varphi_a\left(b\right),\, \psi^{-1}_{\psi_{f\left(a\right)}
\psi^{-1}_{f\left(a\right)} f\varphi_a\left(b\right)}\bar{L}_{\psi_{f\left(a\right)}
\psi^{-1}_{f\left(a\right)}f\varphi_a\left(b\right)}f\left(a\right)\right)\\
    &=  r_\psi\left(f\left(a\right),\, \psi^{-1}_{f\left(a\right)}f\varphi_a\left(b\right)\right)
    = r_\psi\Phi\left(a,b\right).
\end{aligned}
$$
Thus,  $\left(X,\, r_\varphi\right)$ and $\left(Y,\, r_\psi\right)$ are $D$-homomorphic via the map $\Phi$. 

It is a simple check to show that bijectivity of $f$ implies bijectivity of $\Phi.$ 
Other way, assume that $\Phi$ is bijective. Let $a,c\in X$ such that $f\left(a\right) = f\left(c\right)$. Then, 
$$
\begin{aligned}
 \Phi\left(a, b\right)   
 = \left(f\left(a\right), \psi^{-1}_{f\left(a\right)}f\varphi_a\left(b\right)\right)
 = \left(f\left(c\right), \psi^{-1}_{f\left(c\right)}f\varphi_c\varphi^{-1}_c\varphi_a\left(b\right)\right)
 = \Phi\left(c, \varphi^{-1}_c\varphi_a\left(b\right)\right)
\end{aligned}
$$
that implies $a = c$. Now, if $a\in X$, then there exist $x,y\in X$ such that $\Phi\left(x,y\right) =
 \left(a,a\right)$, hence there exists $x\in X$ such that $f\left(x\right) = a$.
\end{proof}

\begin{exa}
If $\left(X,\,\triangleright\right)$ is a left rack, set $\varphi_a = L_a$, for all $a\in X$, by \ref{rem:prop-L_a},  
assumptions in \cref{le:lndsol} are trivially satisfied. Thus, the map $r_\varphi:X\times X\to X\times X$ 
given by $r_\varphi\left(a,\, b\right) = \left(a\triangleright b,\, a\right)$, for all $a,b\in X$, is a solution. 
\end{exa}

\begin{exa}
Let $\left(X,\,+\right)$ be a (not necessarily abelian) group and consider the quandle operation on $X$ given by $a\triangleright b:= -a + b + a$, for all $a, b\in X$. Let $f\in \Aut\left(X,+\right)$, $k\in Z\left(X,\,+\right)$, and set $\varphi_a\left(b\right):= k + f\left(b\right)$, for all $a,b\in X$. 
    Then, $f\in \Aut\left(X,\,\triangleright\right)$ and 
    the identity \eqref{eq:varphi} is trivially satisfied. Hence, by  \cref{le:lndsol}, the map $r_\varphi:X\times X\to X\times X$ given by 
    $$
    r_\varphi\left(a,\,b\right):=
    \left(k + f\left(b\right),\,  -f^{-1}\left(k\right) - b + f^{-1}\left(a\right) + b\right),
    $$
    for all $a,b\in X$, is a bijective and non-degenerate solution on $X$.
    Let us note that the binary operation given by 
   $a\circ b = a + f\left(k\right) + b$, for all $a,b\in X$, endows $X$ of a structure of group with identity $-f\left(k\right)$ and,  for every $a\in X$, the inverse of $a$ is $a^- = - f\left(k\right) - a - f\left(k\right)$. Moreover, the structure $\left(X, +, \circ\right)$ is a (singular) near brace (that is similar to \cite[Example 2.10]{DoRy23}).
\end{exa}

\section{Quandles \& and pre-Lie skew braces}\label{Sec:Quandle&Lie}

\noindent
This section aims to study Yang-Baxter algebras associated to left non-degenerate solutions 
$(X,r)$, namely, structures $(X,m)$ where $m$ a binary operation on $X$ satisfying the identity $mr=m$.
In the first part, we 
show that, for any left-non-degenerate solution, such an algebra exists, and, in general, it is not 
associative. In the second part, we introduce a generalisation of affine quandles using metahomomorphisms (arising from solutions given in \cite{Gu97}) 
and heap endomomorphisms.
Even though definitions  of metahomomorphisms and 
heap endomomorphisms do not contain each other, we show that both of them yield  a Yang-Baxter algebra. 
Furthermore, we prove that some Yang-Baxter algebras lead to right pre-Lie skew braces, namely, algebraic structures including pre-Lie rings.

\subsection{Yang-Baxter algebras}
We note that in this section,  whenever we say algebra we mean a binary algebra in a universal algebra sense,  c.f. \cite{Bergman}.

\begin{defn} \label{qua}
Let $(X,r)$ be a solution of the set-theoretic Yang-Baxter equation.  We say that a pair $(X,m),$ 
where $m:X\times X\to X$, is a Yang-Baxter (or braided) algebra,  if for all $x,y\in X$, $m(x,y)=m(r(x,y)).$ 
\end{defn}

\begin{rem}
Observe that we assume nothing about $m,$ thus $(X,m)$ is in general a magma.
\end{rem}

\begin{rem}
If $(X,m)$ is a Yang-Baxter algebra for some solution $r$ and $\varphi:X\times X\to X\times X$ is 
a $D$-isomorphism, then $(X,m \varphi)$ is a Yang-Baxter algebra for a solution $\varphi^{-1}r\varphi.$
\end{rem}

\begin{lemma}\label{lemmagen-2}
Let $(X,r)$ be a left non-degenerate solution and
$(X, \triangleright_r)$ the shelf associated to $r$. Then, 
if $x\in X$, the binary operation $\bullet$ on $X$ defined by 
\begin{align*}
a \bullet b = \sigma_a\left(b\right)\triangleright_r\left(a\triangleright_r x\right).
\end{align*}
makes $\left(X,\bullet\right)$ a Yang-Baxter algebra.
\end{lemma}
\begin{proof}
Let $a,b\in X$. Since by \cref{le:lndsol} $\tau_a\left(b\right) = \sigma^{-1}_{\sigma_a\left(b\right)}\left(\sigma_a\left(b\right)\triangleright_r a\right)$, 
it follows that
\begin{align*}
\sigma_a\left(b\right)\bullet  \big(\sigma^{-1}_{\sigma_a\left(b\right)}\left(\sigma_a\left(b\right)\triangleright_r a\right)\big)
= \left(\sigma_a\left(b\right)\triangleright_r a\right)\triangleright_r
\left(\sigma_a\left(b\right)\triangleright_r x\right)
=\sigma_a\left(b\right)\triangleright_r\left(a\triangleright_r x\right)
= a \bullet b,
\end{align*}
which is our claim.
\end{proof}

\begin{rem}\label{lemmagen}
    In the special case of a derived solution $\left(X,\, r\right)$, $r\left(a,\,b\right) = \left(b,\,  b\triangleright a\right)$, 
we clearly obtain that for that associated Yang-Baxter algebra $\left(X,\bullet\right)$ it is satisfied the identity
 \begin{align*}
    a\bullet b = b\bullet\left(b\triangleright a\right),
 \end{align*}
    for all $a,b\in X$.
\end{rem}

In the linearized version of the problem Definition \ref{qua} is equivalent to the following statement for set-theoretic solutions. 
Consider a free vector space $V= \mathbb{C}X$ of dimension equal to the cardinality of $X$. 
Let  ${\mathbb B} = \{e_a\}_{a\in X}$ be the basis of $V$ and ${\mathbb B}^* = \{e^*_a\}_{a\in X}$ 
be the dual basis: $e_a^* e_b= \delta_{a,b },$ also  $e_{a,b} := e_a  e_b^*.$
Let also $r \in \mbox{End}(V \otimes V)$ be a solution of the braid equation and 
$ \hat {\mathrm q} \in V \otimes {\cal Q}, $ then the structure ${\cal Q}$ is a Yang-Baxter (or braided) algebra if 
\begin{equation}
r\ (\hat {\mathrm q}\otimes \mbox{id}) ( \mbox{id} \otimes \hat {\mathrm q})=  (\hat {\mathrm q}\otimes \mbox{id})  
(\mbox{id} \otimes \hat {\mathrm q}), \label{funda} 
\end{equation}
where $ \hat {\mathrm q}= \sum_{a\in X} e_a \otimes{\mathrm q}_a$ and $r = \sum r(a,c| b,d) e_{a,b} \otimes e_{c,d}. $ 
If $\cdot: X\times X \to X,$ then
${\mathrm q}: X \to {\cal Q},$ $a \mapsto {\mathrm q}_a,$ is a $(X, \cdot)$ homomorphism. 
The interested reader is referred to e.g.  \cite{FRT, Majid} for a more detailed description. 
This description can be formally  generalized to infinite countable sets, 
 but for continuous sets the  description requires functional analysis and 
kernels of integral operators to describe the solution $r.$

\begin{exa} \label{exa0}
We present below three key examples of set-theoretic solutions of 
the braid equation and their respective Yang-Baxter algebras $(X, \cdot).$
\begin{enumerate}
\item Flip map: $r(a,b)=(b,a)$ and the Yang-Baxter 
algebra is $a\cdot b=b\cdot a.$
\item General set-theoretic solution: $r(a,b)=(\sigma_a(b),\tau_b(a))$
and the Yang-Baxter algebra is $a\cdot b=\sigma_a(b) \cdot \tau_b(a).$
\item Shelf solution: $r(a,b)=(b,b\triangleright a)$ 
and the Yang-Baxter algebra is $a\cdot b=b \cdot (b\triangleright a)$.
\end{enumerate}
\end{exa}
Henceforth, whenever we say set-theoretic 
solution we mean the general set-theoretic solution, $r(a,b)=(\sigma_a(b),\tau_b(a)).$

Let us first recall the well known maps coming from skew braces.  Before we further proceed we recall the definition of skew braces. 

\begin{defn}\cite{Ru05}-\cite{Ru19},  \cite{GuaVen}.
A {\it left skew brace} is a set $B$ together with two group operations $+,\circ :B\times B\to B$, 
the first is called addition and the second is called multiplication, such that for all $ a,b,c\in B$,
\begin{equation}\label{def:dis}
a\circ (b+c)=a\circ b-a+a\circ c.
\end{equation}
If $+$ is an abelian group operation $B$ is called a {\it left brace}.
Moreover, if $B$ is a left skew brace and for all $ a,b,c\in B$ $(b+c)\circ a=b\circ a-a+c \circ a$, then $B$ is called a 
{\it skew brace}. Analogously if $+$ is abelian and $B$ is a skew brace,  then $B$ is called a {\it brace}.
\end{defn}
The additive identity of a skew brace $B$ will be denoted by $0$ and the multiplicative identity by $1$.  
In every skew brace $0=1$.

Let $(X,+, \circ)$ be a skew brace and let $r, \ r^*: X \times X \to X \times X$ 
such that $r (x,y) = (\sigma_{x}(y),\tau_{y}(x)), $  $r^* (x,y) = (\hat \sigma_{x}(y), \hat \tau_{y}(x))$ 
are solutions of the braid equation,  and $r r ^* = \mbox{id}$:
\begin{eqnarray}
&& (1) \quad \sigma_a(b) = -a + a \circ b \nonumber\\
&& (2) \quad  \hat \sigma_a(b)  = a \circ b - a.
\end{eqnarray}

We recall below the underlying Yang-Baxter algebra for the 
corresponding quandles (also examined in \cite{LebVen}).
\begin{enumerate}
\item  For $\sigma_a(b) =  -a + a\circ b$ the corresponding left shelf operation of derived solution takes the simple form, 
\begin{equation}
b \triangleright a = - b +a +b.
\end{equation}
The underlying Yang-Baxter algebra is satisfied  as $a+b = b + (b \triangleright a ),$ i.e.  we define 
$a \bullet b := a+b$ and $(X, +)$ is apparently a group.

\item For $\sigma_a(b) =   a\circ b -a $ the corresponding left shelf action takes the simple form 
\begin{equation}
b \triangleright a = b +a -b.
\end{equation}

Notice that if $+$ is a commutative operation then, $b \triangleright a  =a$ and the twisted map $r_{\varphi}$ 
reduces to the flip map.  Then any commutative binary operation $m$ on $X$ makes $(X,m)$ into Yang-Baxter algebra. 
In particular, $(X,+)$ is a Yang-Baxter algebra.
For not necessarily commutative group, the underlying Yang-Baxter algebra satisfied 
$b +a   =(b \triangleright a ) +b,$ i.e.  
we define $a \bullet b := b+a$ and $(X, \bullet)$ is  a Yang-Baxter algebra.
\end{enumerate}
\noindent {\it Note that in both cases above the map $a \mapsto b \triangleright a$ 
is a bijection and also $a \triangleright a =a,$ 
i.e.  $(X, \triangleright )$ is a quandle.}

We present in what follows some fundamental findings of the present analysis.
We  consider a generalized version of the set-theoretic solution of the braid equation
by introducing some ``$z$-deformation'' \cite{DoiRyb22, DoRy23}.  Indeed, let $z\in X$ be fixed, then we denote 
\begin{equation}
r^{p}(x,y) = (\sigma^p_{x}(y),\tau^p_{y}(x)). \label{genr}
\end{equation}
The superscript $p$ stands for parametric; in our case we consider 
functions of three fixed parameters $z_{1,2},\ z \in X,$  then  
$f^p:=f(z_1,  z_2, z)$ for any function $f$ of the parameters.

Let $(X,+, \circ)$ be a skew brace and $r^{p},\ r^{*p}: X \times X \to X \times X$ are solutions 
of the braid equation, such that  $r^p(x,y) = (\sigma^{p}_{x}(y), \tau^{p}_{y}(x)),$ 
$r^{*p} (x,y) = (\hat \sigma^p_{x}(y), \hat \tau^p_{y}(x))$ \cite{DoRy23}, 
\begin{eqnarray}
& & (1) \quad \sigma_{a}^p(b)= z_1 - a \circ z + a \circ b \circ z_2 \nonumber \\
& & (2) \quad \hat \sigma_a^{p}(b) = a\circ b\circ z_1 - a \circ z +z_2 \label{zdef}
\end{eqnarray}
such that for all $a \in X,$ $a\circ z_2 \circ z_1 - a \circ z = z_2 \circ z_1 - z:= c_1$ 
and $-a\circ z + a \circ z_1 \circ z_2 = -z + z_1 \circ z_2 :=c_2.$
Both solutions satisfy the Yang-Baxter algebra  conditions,
\begin{equation}
a\circ b = \sigma_a^p(b) \circ \tau_b^p(a).
\end{equation}
Also $r^{\hat p}  r ^{*p} = \mbox{id},$ where $r^{\hat p} := r (\hat z_1,\hat z_2 ,  \hat z),$ 
such that $\hat z =z^{-1}, $ $\hat z_{1,2} = z_{2,1} \circ z^{-1}.$ 
In this case the Yang-Baxter algebra is a group.
For the rest of this section we are omitting the superscript $p$ for brevity.

\begin{lemma} \label{lemmaq}
Let $(X, +,\circ)$ be a skew brace and $\sigma_a(b),$ $\hat \sigma_a(b)$ be the bijective 
maps defined in (\ref{zdef}),  also $\tau_{b}(a) = \sigma_a(b)^{-1} \circ a \circ b$ 
and $ \hat \tau_{b}(a) = \hat \sigma_a(b)^{-1} \circ a \circ b$.  Then the corresponding left shelf operations are
\begin{eqnarray}
&&  (1) \quad b \triangleright a: =  \sigma_{b}(\tau_{\sigma^{-1}_a(b)}(a)) = z_1- b \circ z + a\circ z -z_1+ b  \nonumber\\
&& (2) \quad b \triangleright a: =  \hat \sigma_{b}(\hat \tau_{\hat \sigma^{-1}_a(b)}(a)) = b  -z_2 + a\circ z -b\circ z +z_2,  \label{basic2}
\end{eqnarray}
and $(X,  \triangleright)$ is a quandle.
\end{lemma}
\begin{proof}
The proof is immediate after substituting the explicit expressions of the bijecitve maps in the expression 
for $b \triangleright a$ (1), (2) in (\ref{basic2}).  In both cases the map $a \mapsto b \triangleright a$ 
is a bijection and also $a \triangleright a =a,$ 
i.e.  $(X, \triangleright )$ is a quandle.

Notice that even if  $+$ is a commutative operation,  
$b \triangleright a \neq a,$ as opposed to the non-parametric case.
\end{proof}

\subsection{Generalized affine shelves and Yang-Baxter algebras}

\noindent
In this section, we extend the notion of the affine quandle to the non-abelian case. To achieve this, 
we introduce the notions of  metahomomorphisms 
and heap endomorphisms.
Every generalisation of an affine quandle given by 
heap endomorphisms
is also given by some metahomomorphisms which are group endomorphisms. 
The opposite is not true. We show that both metahomomorphisms 
and heap endomorphisms lead to Yang-Baxter algebras on the corrsponding affine quandle solutions. 
In this case, not every operation given by heap endomorphisms is given by some 
metahomomorphism and vice versa. Next, we introduce the notion of the pre-Lie skew brace, 
and observe that, even though both metahomomorphisms and heap endomorphisms lead to 
quandles and Yang-Baxter algebras, the property of heap endomorphisms allows us to 
construct pre-Lie skew braces.

Let us start by recalling the notions below.
\begin{defn}
  Let $\left(X,+\right)$ be a (not necessarily abelian) group and $f:X\to X$ a map.   
    \begin{enumerate}
        \item $f$ is said to be a \emph{heap endomorphism of $\left(X,+\right)$}  if
        \begin{align*}
            \forall\ a,b,c\in X\qquad   
            f(a-b+c)=f(a)-f(b)+f(c).
        \end{align*}
        \item \cite[Definition 1]{Gu97} $f$ is said to be a \emph{metahomomorphism of $\left(X,+\right)$} if
        \begin{align*}
        \forall\ a,b\in X\qquad f\left(a + b - f\left(a\right)\right) = f\left(a\right) + f\left(b\right)  - f^2\left(a\right). 
    \end{align*}
    \end{enumerate}
\end{defn}

\begin{exa}
  Let $\left(X,+\right)$ be a group. 
    \begin{enumerate}
        \item Any endomorphism of $\left(X,+\right)$ is both a heap endomorphism and a 
        metahomomorphism of $\left(X,+\right)$.
        \item The map $\iota$ defined by $\iota\left(a\right) = -a$, for every $a\in X$, is a 
        metahomomorphism of $\left(X,+\right)$ that is neither a homomorphism nor a heap endomorphism if $\left(X,+\right)$ is not abelian.
        \item The constant map of value $k\in X$, namely, the map $f$ defined by $f\left(a\right) = k$, for every $a\in X$, is both a heap endomorphism and a metahomomorphism of $\left(X,+\right)$. 
    \end{enumerate}
\end{exa}

\begin{rem}\label{rem:2metend}
   Let $(X,+)$ be a group and $f,\hat{f}:X\to X$ maps such that $\hat{f}(x):=f(x)-f(0)$, for all $x\in X$. Then,
    \begin{enumerate}
        \item $f$ is a heap endomorphism of $(X,+)$ if and only if $\hat{f}$ is a homomorphism of $(X,+)$;
      
        \item if $f$ is a metahomomorphism of $(X,+)$, then $\hat{f}$ is metahomomorphism of $(X,+)$ such that $f(0) = 0$, namely, $f$ is a \emph{unitary} metahomomorphism. More generally, any metahomomorphism can be obtained from a unitary one (cf. \cite[Theorem 4]{Gu97}). 
    \end{enumerate}
    Moreover, note that any heap endomorphism $f$ such that $f(0)\in Z(X,+)$ is a metahomomorphism.
    
\end{rem}

We are now ready to introduce generalisations of affine quandles.
We include the first example as it was inspired by the affine quandle.  Compare with \cite[(1.1)]{hou} 
and \cite[Theorem 5.9]{BrBrRySa23}.

\begin{exa}\label{ex:op:1}
Let $f$ be a heap endomorphism of a group $(X,+)$. Then the following operation
\begin{equation} 
    a\triangleright_t b=f(b)-f(a)+a
\end{equation}
is a spindle, and if $f$ is bijective it is a quandle.
More generally, if $f:X\to X$ is a bijective map such that $f(0) = 0$, it is easy 
to check that $\left(X, \triangleright_t\right)$ is a quandle if and only if 
$f\left(a - f\left(b\right) + b\right) = f\left(a\right) - f^2\left(b\right) + f\left(b\right)$ holds, for all $a,b\in X$. 
Note that, in the case of $f\in\Aut\left(X,+\right)$, the structure $\left(X,  \triangleright_t \right)$ 
coincides with  (right) quandles provided in \cite[p. 42]{Jo82}.
\end{exa}

The second is motivated by the connection to metahomomorphisms and particular solutions provided in \cite{Gu97}, and  generalizes the one acquired from braces.
\begin{exa}\label{ex:op:2}
Let $(X,+)$  be a group and $f:X\to X$ be a metahomomorphism. Then the following operation
\begin{equation}
    a\triangleright_r b = a + f\left(b\right) - f\left(a\right)
\end{equation}
is a spindle, and if $f$ is bijective it is a quandle. It is easy to check that if $f:X\to X$ is a 
bijective map such that $f\left(0\right) = 0$, then 
and $\left(X, \triangleright_r\right)$ is a quandle if and only if $f$ is a metahomomorphism.
\end{exa}

The third example also generalizes the one acquired from braces,  see \eqref{basic2}.

\begin{exa}\label{ex:op:3} Let $(X,+)$ be group and $f:X\to X$ a 
metahomomorphism of the group $(X,+_{op}),$ 
where $a+_{op} b:=b+a$, for all $a,b\in X$. Then the following operation
    \begin{equation}
    a\triangleright_s b = - f\left(a\right) + f\left(b\right) + a
\end{equation}
is a spindle, and if $f$ is bijective it is a quandle.
It is easy to check that if $f:X\to X$ is a bijective map such that $f\left(0\right) = 0$, then 
$\left(X, \triangleright_s\right)$ is a quandle if and only if $f$ is a metahomomorphism of $(X,+_{op})$.
\end{exa}

\begin{rem}
Observe that if $(X,+)$ is 
an abelian group and $f\in \mathrm{Aut}(X,+)$, 
then all the operations from examples \ref{ex:op:1}, \ref{ex:op:2} and \ref{ex:op:3} 
coincide, and we acquire a quandle known in literature as \emph{affine quandle}. 
\end{rem}

\begin{lemma}\label{lem:3stories}
If $(X,+)$ is a group and $f$ a heap endomorphism, then there exist $k\in X$ and 
$l,r\in \mathrm{End}(X,+)$ such that, for all $x\in X$, $f(x)=k+l(x)=r(x)+k$. Moreover, the following hold:
\begin{enumerate}
\item For all  $a,b\in X$, $f(b)-f(a)+a=r(b)-r(a)+a,$ that is heap endomorphism allows us 
to define the spindle $\triangleright_t$ from \cref{ex:op:1} for a metahomomorphism $r$.
\item For all  $a,b\in X$, $a+f(b)-f(a)=a+r(b)-r(a),$ that is heap endomomorphism 
allows us to define the spindle $\triangleright_r$ from \cref{ex:op:2} for a metahomomorphism $r$.
\item For all  $a,b\in X$, $-f(a)+f(b)+a = -l(a)+l(b)+a,$ that is heap endomorphism 
allows us to define the  spindle $\triangleright_s$ from \cref{ex:op:3} for a metahomomorphism $l$.
\end{enumerate}
\end{lemma}

\begin{rem}
  Let $\left(X,+\right)$ be a group and $f:X\to X$ a map. Then, the map $t:X\times X\to X\times X$ defined by
    $$
        t\left(a,b\right) 
    = \left(- b - f\left(a\right) + a, f\left(a\right)\right),
    $$
    is a solution of the braid equation if and only if $f\left(a - f\left(b\right) + b\right) = f\left(a\right) - f^2\left(b\right) + f\left(b\right)$,
    for all $a,b\in X$, that is the same identity found in \cref{ex:op:1}.  Moreover, $t$ is left non-degenerate and 
the shelf operation associated with it is 
the operation on $X$ given by $a\triangleright_t b = f\left(b\right) - f\left(a\right) + a$, for all $a,b\in X$, 
thus the shelf $\left(X, \triangleright_t\right)$ is the same in \cref{ex:op:1}.
\end{rem}

\begin{rem}\label{rem-Gu+}
Let  $\left(X,+\right)$ be a group and $f:X\to X$ a map. Then, the map $r:X\times X\to X\times X$ defined by
$$
r\left(a,b\right) = \left(a + b - f\left(a\right), f\left(a\right)\right)
$$
is a solution  of the braid equation if and only $f$ is a metahomomorphism  of $\left(X,+\right)$, see \cite[Theorem 1]{Gu97}.
Analogously, one can check that the map $s:X\times X\to X\times X$ defined by
$$
s\left(a,b\right) = \left(-f\left(a\right) + b + a, f\left(a\right)\right)
$$
is a solution  of the braid equation if and only $f$ is a metahomomorphism of the 
group $\left(X,+_{op}\right)$. 

Moreover, let us observe that the solutions $r$ and $s$ are left non-degenerate, and the shelves associated with them are 
\begin{align*}
    a\triangleright_r b = a + f\left(b\right) - f\left(a\right)
    \qquad\text{and}\qquad
     a\triangleright_s b = - f\left(a\right) + f\left(b\right) + a,
\end{align*}
respectively, thus they coincides with \cref{ex:op:2} and \cref{ex:op:3}, respectively.
\end{rem}

\begin{lemma} \label{pro0}
 Let $(X, \triangleright_t)$ be the quandle from \cref{ex:op:1}, i.e., $a\triangleright_t b = f(b) - f(a) + a$, for all $a,b\in X$. 
Let us define the following binary operation on $X$
\begin{align}\label{bullet:1}
a\bullet_t b := -f^2(a) + f(a) - f(b) + b.
\end{align}
Then, $a\bullet_t b = b \bullet_t (b\triangleright_t a)$, i.e. $(X,\bullet_t)$ is a Yang-Baxter algebra for $r_{\triangleright_t}$.
\end{lemma}

\begin{lemma}\label{prop1}
Let $(X, \triangleright_r)$ be the quandle from \cref{ex:op:2},
i.e., $a\triangleright_r b := a + f(b) - f(a)$, for all $a,b \in X$, for $f$ 
being bijective metahomomorphism or a heap 
homomorphism $(X,+)$. Let us define the following binary operation on $X$
\begin{align}\label{bullet:2}
a \bullet_r b := b + f(a). 
\end{align}
Then, $a\bullet_r b = b \bullet_r (b\triangleright_r a)$, i.e. $(X,\bullet_r)$ is a Yang-Baxter algebra for $r_{\triangleright_r}$.
Similarly, for the quandle $(X, \triangleright_s)$ in \cref{ex:op:3},
i.e., $a\triangleright_s b := -f(a) + f(b) + a$, for all $a,b \in X$, for $f$ being bijective metahomomorphism  
or a heap endomorphism of $(X,+_{op})$, the binary operation
\begin{align}\label{bullet}
a \bullet_s b := f(a) + b
\end{align}
satisfies $a\bullet_s b = b\bullet_s (b\triangleright_s a),$ i.e.  $(X,\bullet_s)$ is a 
Yang-Baxter algebra for $r_{\triangleright_s}$.
\end{lemma}
\begin{proof}
We  compute, for the quandle $(X, \triangleright_r)$,
\begin{eqnarray}
b\bullet_r (b \triangleright_r a) =  b\triangleright_r a + f(b) = b + f(a) - f(b) + f(b) = a\bullet_r b. 
\nonumber
\end{eqnarray}
Similarly, the second part of the statement holds for the quandle $(X, \triangleright_s)$ and the operation $\bullet_s$.
\end{proof}

\begin{rem}\label{rem:hk}
Let $\left(X,+\right)$ be a group, $\left(X,\bullet\right)$ a Yang-Baxter algebra, and $h,k\in X$. 
Then, the binary operation on $X$ defined  by $a\cdot_{h,k} b:= h + a\bullet b + k$, for all $a,b\in X$, 
makes $(X,\cdot_{h,k})$ a Yang-Baxter algebra.
\end{rem}

\begin{exa}
From the expression of Lemma \ref{lemmaq}  and expressions (\ref{ex:op:3}), (\ref{bullet}) we conclude:
\begin{enumerate}
\item $a \bullet_s b = a\circ z -z_1 +b .$
\item  $a \bullet_r b = b -z_2 + a\circ z.$
\end{enumerate}
\end{exa}

\subsection{Pre-Lie skew braces}

\noindent Motivated by the notion of pre-Lie algebras (also studied under the name chronological algebras) 
\cite{Chrono, PreLie, Pre1} (see also \cite{Bai, Manchon} for a recent  reviews) we introduce a novel algebraic 
structure called  a {\it pre-Lie skew brace} to describe the underlying 
structures associated to certain set-theoretic solutions of the braid equation.  In fact, we identify families 
of affine quandles that generate pre-Lie skew braces.

\begin{defn}\label{def1:prelie}
    Let (X,+) be a group and $\bullet:X\times X\to X$ be a binary operation. We say that the triple $(X,+,\bullet)$
 is a {\bf right pre-Lie skew brace} if for all $a,b,c\in X$ the following hold:
    \\
{\it 1. Distributivity}
\begin{eqnarray}
 a \bullet (b +c) = a\bullet b -a\bullet 0 +a\bullet c \quad \& \quad  (a+b ) \bullet c = 
a\bullet c -0\bullet c +b\bullet c. \label{distr} 
\end{eqnarray}
2. Right pre-Lie condition
\begin{eqnarray}
& &  (a \bullet b ) \bullet c - a \bullet (b \bullet c) =  (a \bullet c ) \bullet b - a \bullet (c \bullet b)\label{prelie:1}.  
\label{prelie}
\end{eqnarray}
Moreover, we say that $(X,+,\bullet)$ is a {\bf left invertible right pre-Lie skew brace}, if $(X,\bullet)$ has a 
left identity, i.e. there exists $e\in X$ such that for all $b\in X$ $e\bullet b=b$, and for all $b\in X$ 
there exists a unique left inverse $b^{-1}\in X$ such that $b^{-1}\bullet b=e.$
\end{defn}

We note that the distributive condition in Definition \ref{def1:prelie} appears naturally in skew braces, also the right pre-Lie condition is satisfied in Definition \ref{def1:prelie}, hence the name 
{\it right pre-Lie skew braces} is well justified for the structure described in Definition \ref{def1:prelie}.

\begin{rem} \label{rempre}
Observe that in the case that $(X, +)$ is an abelian group (pre-Lie brace), we get a pre-Lie affgebra introduced in \cite{Brz7} Definition 3.8.
\end{rem}

\begin{exa}
Every pre-Lie ring is a pre-Lie brace and any two-sided nearring is a skew pre-Lie brace.
\end{exa}

\begin{exa}
If $(R,+,\cdot)$ is a ring, then $(R,+,\circ),$ where $a\circ b=a+b+a\cdot b$ 
is the adjoint operation of $R$ is a pre-Lie brace. In particular, if $(R,+,\circ),$ 
is a brace, then $(R,+,\circ)$ is a left invertible right pre-Lie brace.
\end{exa}

\begin{rem}
    Let $(X,+)$ be a group and $\bullet:X\times X\to X$ be a binary operation such that \eqref{distr} holds. 
Then $(X,+,\bullet)$ is a right pre-Lie skew brace if and only if for all 
    $a,b,c\in X$,
    \begin{align}
       -_{op}\, a \bullet (b \bullet c) +_{op} \, (a \bullet b ) \bullet c  
          =-_{op}\,a \bullet (c \bullet b) +_{op}\, (a \bullet c ) \bullet b,\label{eq:op+}
    \end{align}
 where $a+_{op}b=b+a$ for all $a,b\in X$.
\end{rem}

\begin{thm}\label{thm:right-pre-Lie}
Let $(X,+)$ be a group and $\bullet_s,\bullet_r:X\times X\to X$  as given in \cref{prop1} 
for some bijective  $f:X\to X.$ Then $(X,+,\bullet_s)$ and $(X,+_{op},\bullet_r)$
are left invertible right pre-Lie skew braces.
\end{thm}
\begin{proof}
Let us show that $(X,+,\bullet_s)$ is a left invertible right pre-Lie skew brace. One can easily check that 
$e:=f^{-1}(0)$ is a unique identity and 
$b^{-1}:=f^{-1}(e-b)$ is the unique left inverse. Thus it is enough to check that \eqref{distr} and 
\eqref{prelie:1} hold. The equality \eqref{distr} is a simple check. For \eqref{prelie:1} let $a,b,c\in X$, then 
 \begin{align*}
        (a\bullet_s b)\bullet_s c - a \bullet_s (b\bullet_s c)=   f^2(a) -f(0) -f(a) =
    (a \bullet_s c) \bullet_s b  - a \bullet_s (c \bullet_s b),
\end{align*}
i.e., the right pre-Lie condition holds. The proof for $(X,+_{op},\bullet_r)$ is analogous,
 just instead of \eqref{prelie:1}, we show \eqref{eq:op+}.
\end{proof}

\begin{rem}  Recall that for the pre-Lie skew braces above: $ a \bullet b = b \bullet (b \triangleright a),$ 
where the $\triangleright $ action is given in \cref{ex:op:3}. Because of the existence of the left inverse it follows:
\begin{equation}
b \triangleright a = (b^{-1} \bullet a) \bullet b. \label{inv}
\end{equation} Explicit computation of the RHS of expression \eqref{inv}  confirms \eqref{inv}.
\end{rem}

Since right pre-Lie skew braces with underlying abelian group (pre-Lie braces) are right pre-Lie affgebras, by Remark \ref{rempre}, the following Lemma and Proposition can be found in \cite{Brz:Lietruss, Brz7}.


\begin{lemma}\label{lem:Lie:truss}
Let $(P,+,\bullet)$ be a pre-Lie brace.
Then a triple $(P,[-,-,-],\{-,-,-\})$ is a Lie truss, where 
$$
[a,b,c]=a-b+c\quad \&\quad \{a,b,c\}:=a\bullet c-c\bullet a+ b,
$$ 
for all $a,b,c\in P$.
\end{lemma}
\begin{proof}
It is a simple check that $(P,[-,-,-],\{-,-,-\})$ satisfies conditions of \cite[Definition 3.1]{Brz:Lietruss}
\end{proof}

\begin{pro}\label{pro:Liering}
Let $(P,+,\bullet)$ be a pre-Lie brace.
Then for all $o\in P,$ $(P,+_o,[-,-])$ is a Lie ring, where 
$$
a+_ob=a-o+b\quad \&\quad [a,b]=a\bullet b-b\bullet a+o\bullet a- a\bullet o+b\bullet o-o\bullet b-o,
$$
for all $a,b\in P.$ We will denote the Lie ring associated to pre-Lie brace $P$ in element $o$, by $P(o).$
\end{pro}
\begin{proof}
Immediately follows by \cref{lem:Lie:truss} and \cite[Proposition 3.6]{Brz:Lietruss}.
\end{proof}

\begin{exa}
Let $(R,+,\cdot)$ be a ring. Then $(R,+,\cdot)$ is also a pre-Lie brace, and $R(0)$ 
is a Lie ring with a Lie bracket being the usual commutator.
\end{exa}

In the following proposition, we show that some operations in \cref{rem:hk} determine right pre-Lie skew brace starting from given right pre-Lie skew braces as, for instance, those provided in \cref{thm:right-pre-Lie}.
\begin{pro}\label{prop3}
Let $(X,+)$ be a group and $f:X\to X$ a bijective . If $h\in X$, then $(X, +, \cdot_{h,0})$ is a left invertible right pre-Lie skew brace, where 
$a\cdot_{h,0} b:= h + a\bullet_s b =  h + f(a) + b$, for all $a,b\in X$. Similarly, if $k\in X$, then, $(X, +_{op}, \cdot_{0,k})$ 
is a left invertible right Pre-Lie skew brace, where 
$a\cdot_{0,k} b:= a\bullet_r b + k = b + f(a) + k$, for all $a,b\in X$.
\end{pro}
\begin{proof}
Let us prove the statement for the structure $(X, +,\cdot_{h,0})$, that  for $(X, +_{op}, \cdot_{0,k})$ is analogous. 
Let $a,b,c\in X$. Then,
$$
a\cdot_{h,0}(b + c) 
= h + f(a) + b - (h + f(a)) + h + f(a) + c
= a\cdot_{h,0} b - a\cdot_{h,0}0 + a\cdot_{h,0} c 
$$
and 
$$
(a + b)\cdot_{h,0} c 
=  h + f(a) + c - (h + f(0) + c) + h + f(b) + c
=  a\cdot_{h,0} c - 0\cdot_{h,0}c + b\cdot_{h,0}c,   
$$
hence distributivity laws are satisfied.
Furthermore,
\begin{align*}
(a \cdot_{h,0} b)\cdot_{h,0} c  &- a\cdot_{h,0} (b\cdot_{h,0} c)
= h + f\big(h + f(a)\big) - f(0) - h - f(a) - h\\ 
&= (a\cdot_{h,0}c)\cdot_{h,0} b  - a\cdot_{h,0}(c\cdot_{h,0}b),
\end{align*}
hence the right pre-Lie condition holds.

For the left invertibility one can easily check that $e=f^{-1}(-h)$ 
is the left identity and, for every  
$b\in X$, $b^*:=f^{-1}(-h+e-b)$ is the left inverse of $b$.
\end{proof}

\begin{exa}\label{ex:trivial}
Let $(X,+,\bullet_r)$ be a pre-Lie skew brace introduced in \cref{prop1}, such that $(X,+)$ is abelian. 
Then, for every $o\in X,$ the Lie ring $X(o)$ defined in \cref{pro:Liering} has a zero Lie bracket.
 Indeed for all $a,b\in G,$ $a\bullet_r b=f(a)+b,$ since $+$ is abelian, and
$$
[a,b]=f(a)+b-f(b)-a+f(o)+a-f(a)-o+f(b)+o-f(o)-b-o=-o,
$$
where $o$ is the neutral element of the group $(X,+_o).$
\end{exa}

\section{Deformed braided algebras \& magmas} \label{sec:4}

\noindent Motivated by the definition of braided groups and braidings in \cite{chin} 
as well as the relevant work presented in \cite{GatMaj} and the generic definition of 
deformed braided algebras and deformed braidings 
that contain extra fixed parameters \cite{DoiRyb22, DoRy23} we 
provide a relevant definition associated to magmas.  

It is useful for the following definition to fix the arbitrary invertible maps for any set $X:$
\begin{eqnarray}
r: X \times X \to X \times X,\qquad (x,y)\mapsto (\sigma_x(y), \tau_y(x))\\
\xi: X \times X \to X \times X, \qquad(x, y) \mapsto (f_x(y),  g_y(x))\\
\zeta: X \times X \to X \times X, \qquad  (x,y)\mapsto (\hat f_x(y),  \hat g_y(x)).
\end{eqnarray}

\begin{defn} \label{def0} \cite{DoRy23} 
Let $(X,m)$ be a group. A map $r$ is called  a $\xi,\zeta$-deformed braiding operator 
(and the group is called $\xi,\zeta$-deformed braided group) if for all $x,y,w \in X:$
\begin{enumerate}
\item $m(x, y) =m (r (x,  y)).$
\item $  \xi (m \times \mbox{id}_X)(x,y,w)= (\mbox{id}_X \times m) (r\times \mbox{id}_X)  (\mbox{id}_X \times r) (x,y,w).$
\item $ \zeta(\mbox{id}_X \times m)(x,y,w)= (m \times \mbox{id}_X) (\mbox{id}_X \times r)  (r\times \mbox{id}_X)(x,y,w).$
\end{enumerate}
\end{defn}
In the special case where $r = \zeta= \xi,$ we have a usual braiding  and braided group.

\begin{lemma} \label{lemmA} \cite{DoRy23}  Let $(X,\circ)$ be a deformed braided group and 
the map $r: X \times X \to X \times X$ be a an $\xi,\zeta$-deformed braiding operator, then 
$r$ satisfies the braid equation.
\end{lemma}
\begin{proof}
From the braid equation \eqref{eq:braid}, with reference to \eqref{eq:braidI}, 
\eqref{eq:braidII}, and \eqref{eq:braidIII}, let us set
\begin{align*}
&\big(\sigma_{\sigma_x(y)}(\sigma_{\tau_{y}(x)}(w)),
\tau_{\sigma_{\tau_y(x)}(w)}(\sigma_x(y)), \tau_w(\tau_y(x))\big) =  
\left(a_1,b_1,c_1\right)
\\
&\big(\sigma_x(\sigma_y(w)), \sigma_{\tau_{\sigma_y(w)}(x)}(\tau_w(y)),
\tau_{\tau_w(y)}(\tau_{\sigma_y(w)}(x))\big)  
=  
\left(a_2,b_2,c_2\right) 
\end{align*}
for all $x,y,w\in X$. Via condition (2) of \cref{def0}, $~\sigma_x(\sigma_y(w) )= 
f_{x\circ y}(w),$ then from condition (1),
$\sigma_{\sigma_x(y)}(\sigma_{\tau_{y}(x)}(w)) = \sigma_x(\sigma_y(w))$,
hence \eqref{eq:braidI} holds and $a_1 = a_2$. Via condition (3) of \cref{def0},
$~\tau_w(\tau_y(x)) = \hat g_{y\circ w}(x),$ then from condition (1), 
$\tau_{\tau_w(y)}(\tau_{\sigma_y(w)}(x)=\tau_w(\tau_y(x))$,
hence \eqref{eq:braidIII} holds and $c_1 = c_2$.
By using many times condition (1) of \cref{def0} it follows that $a_1\circ b_1\circ c_1 = 
x\circ y\circ w = a_2\circ b_2\circ c_2$, 
hence $b_1 = b_2$, i.e., \eqref{eq:braidII} holds.
\end{proof}

We note that given the existence of an invertible, non-degenerate, set-theoretic solution of the braid equation,  
together with the group $(X, \circ),$ the construction of a skew brace consistently follows 
(see details of such a construction in
\cite[Section 2.2]{DoiRyb22}.)

Motivated by recent observations on deformed solutions and 
connections to Yang-Baxter algebras from section 2, we introduce the following notion of a generalized 
braiding operator designed for Yang-Baxter algebras. Observe that, we put a condition on the form of 
the braiding, but this condition naturally arises from the construction of a quandle/shelf solution.

\begin{defn} \label{defA}
Let $(X, m)$ be a magma and for all $ x, y \in X, $
\begin{enumerate}[{(i)}]
\item $r(x,y) = (y,  \tau_y(x)),$ $\xi (x, y) = (y,  g_y(x))$ and 
$\zeta(x,y) = (y, \hat g_y(x)),$ or 
\item  $r(x,y) = (\sigma_x(y),  x),$ $\xi (x, y) = (f_x(y), x)$ and 
$\zeta(x,y) = (\hat f_x(y),  x).$
\end{enumerate}
The map $r$ is called  an 
$S$-deformed braiding operator (and the magma is called $S$-deformed braided magma) if for all 
$x,y,w \in X:$
\begin{enumerate}
\item $m(x, y) =m (r (x,  y)).$
\item $  \xi (m \times \mbox{id}_X)(x,y,w)= (\mbox{id}_X \times m) (r\times \mbox{id}_X)  (\mbox{id}_X \times r) (x,y,w).$
\item $ \zeta(\mbox{id}_X \times m)(x,y,w)= (m \times \mbox{id}_X) (\mbox{id}_X \times r)  (r\times \mbox{id}_X)(x,y,w).$
\end{enumerate}
\end{defn}

We note that Definition \ref{defA} can bee seen as special cases of Definition 
\ref{def0},  however the underlying algebraic structure in \ref{defA}  does not have to be a group as in \ref{def0}.
In the special case where $r = \zeta= \xi,$ we have $S$-braidings and $S$-type braid algebras.

\begin{lemma}\label{le:S-def1}
Let $(X,\bullet)$ be an $S$-deformed braided magma and the map 
$r: X \times X \to X \times X,$ such that $r(x,y) = (y,  \tau_y(x))$ be an 
$S$-deformed braiding operator, then 
$r$ satisfies the braid equation.
\end{lemma}
\begin{proof}
From the braid equation:
\begin{eqnarray}
&& (r\times \mbox{id}_X) (\mbox{id}_X \times r)(r\times \mbox{id}_X) (x,y,w) =  
\big ( w,\tau_w(y), \tau_w(\tau_y(x))\big ) \label{r1a}\\
&&  (\mbox{id}_X \times r)(r\times \mbox{id}_X) (\mbox{id}_X \times r)(x,y,w)  = 
 \big (w, \tau_w(y),\tau_{\tau_w(y)}(\tau_{w}(x))\big). \label{r2a}
\end{eqnarray}
By comparing  expressions (\ref{r1a}) and (\ref{r2a}) we conclude: \\
Via condition (3) of Definition \ref{defA},
 $~\tau_w(\tau_y(x)) = \hat g_{y\bullet w}(x),$ then from condition (1):
\begin{equation}
\tau_{\tau_w(y)}(\tau_{w}(x)=\tau_w(\tau_y(x)). \label{f2a} \nonumber\\
\end{equation}
The latter condition is the left self-distributivity \eqref{prop2-La} and the left shelf action is $b\triangleright a:= \tau_{b}(a)$.

\noindent Conditions (1) and (3) of Definition \ref{defA} suffice to show 
that $r$ satisfies the braid equation.
\end{proof}

\begin{lemma}  
Let $(X,\bullet)$ be an $S$-deformed braided magma and the map 
$r: X \times X \to X \times X,$ such that $r(x,y) = (\sigma_x(y),  x),$ 
be an $S$-deformed braiding operator, 
then $r$ satisfies the braid equation.
\end{lemma}
\begin{proof}
The proof is similar to that of \cref{le:S-def1} by considering the 
left shelf action given by $a\triangleright b:= \sigma_a(b)$.\\
From the braid equation:
\begin{eqnarray}
&& (r\times \mbox{id}_X) (\mbox{id}_X \times r)(r\times \mbox{id}_X) (x,y,w) =  
\big ( \sigma_{\sigma_x(y)}(\sigma_{x}(w)),\sigma_x(y), x\big ) \label{r1b}\\
&&  (\mbox{id}_X \times r)(r\times \mbox{id}_X) (\mbox{id}_X \times r) (x,y,w)  =  
\big (\sigma_x(\sigma_y(w)), \sigma_x(y), x\big). \label{r2b}
\end{eqnarray}
By comparing  expressions (\ref{r1b}) and (\ref{r2b}) we conclude: \\
Via condition (2) of Definition \ref{defA}, $~\sigma_x(\sigma_y(w) )= f_{x\bullet y}(y),$ then from condition (1):
\begin{equation}
\sigma_{\sigma_x(y)}(\sigma_{x}(w)) = \sigma_x(\sigma_y(w)). \label{f1b} \nonumber\\
\end{equation}
The latter condition is  the left self-distributivity \eqref{prop2-La} and the left shelf action is 
$a\triangleright b:= \sigma_a(b)$.

\noindent Conditions (1) and (2) of Definition \ref{defA} 
suffice to show that $r$ satisfies the braid equation.
\end{proof}

Definitions \ref{def0},  \ref{defA} encode essentially the notion of a quantum algebra,  
i.e.  a (quasi)-bialgebra. 
Several examples associated to structures of Definition \ref{def0} are presented in \cite{DoiRyb22, DoRy23}, 
whereas the underlying quantum algebra is studied in \cite{Doikou1, DoGhVl, DoiRyb22}.
We present below several examples related to the structures of the Definitions  above.

\begin{exa} \label{ex:4:1} Let $(X, + ,\circ)$ be a skew brace,  $ \sigma_a(b)= -a+a\circ b, $  
$\tau_b(a) =\sigma_a(b)^{-1}\circ a \circ b,$  and
$r (a,b) = (\sigma_a(b), \   \tau_b(a)).$ In this case $\sigma_a(\sigma_b(c)) = \sigma_{a\circ b}(c),$
 similarly  $\tau_{c}(\tau_b(a)) = \tau_{b\circ c}(a),$ i.e.
$\sigma_{a} =f_a= \hat f_a.$ Thus $r$ is a braiding and $(X, \circ)$ is a braided group.
\end{exa}

\begin{exa}\label{ex:4:2}  Let $(X,+,  \circ)$ be a skew brace,  $ \sigma_a(b)= -a\circ z+a\circ b\circ z, $ 
$z\in X$ is a fixed element,   
$\tau_b(a) =\sigma_a(b)^{-1}\circ a \circ b,$  and
$r (a,b) = (\sigma_a(b), \   \tau_b(a)).$ In this case $\sigma^z_a(\sigma^z_b(c)) = \sigma^{z\circ z}_{a\circ b}(c),$
similarly $\tau^z_{c}(\tau^z_b(a)) = \tau^z_{b\circ c}(a),$ i.e.
$f_a = \sigma^{z\circ z}_{a}$ and $\hat f_a = \sigma_a^z.$ Then $r$ 
is a deformed braiding and $(X, \circ)$ is a deformed braided group.
\end{exa}

\begin{exa}\label{ex:4:3}
Let $(X, +)$ be a group,
$\tau_b(a) = - b +a +b$ and
$r (a,b) = (b, \  \tau_b(a)).$ In this case $\tau_{c}(\tau_b(a)) = \tau_{b + c}(a),$ i.e.
$g_a = \hat g_a = \tau_a.$ We obtain that $r$ is an $S$-type braiding and $(X, +)$ is an $S$-type braided group.
\end{exa}

\begin{exa} \label{ex:4:4}
Let $(X,+)$ be a group, $(X, \triangleright_s)$ the quandle in \cref{ex:op:3}, and 
$\tau_b(a):= b\triangleright_s a = -f(b) + f(a) + b$, for all $a,b\in X$, with $f$ a bijective heap endomorphism of $(X,+)$. 
We also consider the binary operation on $X$ in \cref{prop1} such that $a \bullet b = f(a) + b$, for all $a,b\in X$. In this case 
$\tau_{c}(\tau_b(a)) = \tau_{b\bullet c}(a\bullet e)$, i.e.
$\hat g_b(a) = \tau_b(a \bullet e),$  where $e = f^{-1}(0)$ is the left neutral element in $(X, \bullet).$ 
Then, $r (a,b) = (b, \  \tau_b(a))$ is an $S$-deformed braiding and $(X, \bullet)$ is an $S$-deformed braided magma.
\end{exa}


\section{Quasi triangular (quasi)-Hopf algebras} \label{sec:5}

\noindent We discuss in this section the Yang-Baxter algebras associated to rack type and set-theoretic solutions 
as quasitriangular (quasi)-Hopf algebras \cite{Drinfel'd, Drinfel'd2}. 
The quasi Hopf algebra for set-theoretic solutions has been partly discussed in \cite{Doikou1, DoGhVl, DoiRyb22}
(see also \cite{EtScSo99} and  \cite{Andru} in connection with pointed Hopf algebras),  whereas a discussion 
of bialgebras associated to racks is presented in \cite{LebMan} (see also \cite{Lebed, braceh} on Hopf algebras in 
connection to braces \cite{Ru05}--\cite{Ru19}, \cite{GuaVen}).   

We start our analysis with the rack and quandle algebras and the construction of the associated universal 
${\cal R}$-matrix.
We then extend the algebra to the {\it decorated rack algebra} and via a suitable admissible Drinfel'd twist 
we construct the corresponding universal ${\cal R}-$matrix. We note that in this section we consider finite or infinite countable sets.

\subsection{The rack \& quandle  algebras}

\noindent We first define the {\it rack} 
and {\it quandle} algebras 
(see also \cite{Andru, LebMan}, 
where part of the defined algebra below is studied).

\begin{defn} \label{def1} (The rack algebra)
Let $(X, \triangleright)$ be a finite magma, or such that $a\triangleright$ is surjective, for every $a\in X$.
We say that the unital, associative algebra ${\cal A},$ over a field $k$ 
generated by indeterminates  ${1_\cal A}$ (the unit element), $~q_a, \ q^{-1}_a,\ h_a\in {\cal A}$ 
($h_a = h_b \Leftrightarrow a =b$)
and relations, $a,b \in X:$
\begin{eqnarray}
q_a^{-1} q_a= q_a q_a^{-1} = 1_{\cal A},\quad q_a q_b = q_b  q_{b \triangleright a}, \quad  h_a  h_b =\delta_{a, b} h_a^2, \quad q_b  h_{b\triangleright a} = h_a q_b, \label{qualg}
\end{eqnarray}
is a {\it rack} algebra.
\end{defn}

\begin{defn}  \label{qualgd} (The quandle algebra)
A rack algebra is called a quandle algebra if there is a left quasigroup $(X, \bullet)$ such that for all $a,b \in X$ $a\bullet b = b\bullet (b \triangleright a).$
\end{defn}

The following proposition fully justifies the appellation rack and quandle algebras.

\begin{pro} \label{proro1}
\label{qua1} Let ${\cal A}$ be the rack algebra,
then $c \triangleright (b \triangleright a) = (c \triangleright b)\triangleright (c \triangleright a),$ and $a\triangleright $ is bijective for all $a\in X,$ i.e. $(X, \triangleright)$ is a rack. If ${\cal A}$ is the quandle algebra, then in addition $a \triangleright a =a,$ i.e. $(X, \triangleright )$ is a quandle.
\end{pro}
\begin{proof}
We compute $h_a q_b q_c$ using the associativity of the rack algebra:
\begin{eqnarray}
&& h_a q_b q_c =q_b h_{b\triangleright a} q_c = q_b q_c h_{c \triangleright(b \triangleright a)} = 
q_c q_{c \triangleright b}h_{c \triangleright(b \triangleright a)}, \label{w1}\\
&& h_a q_b q_c =h_a q_c q_{c\triangleright b}  = q_c h_{c\triangleright a} q_{c \triangleright b } = 
q_c q_{c \triangleright b}h_{(c \triangleright b) \triangleright ( c\triangleright a)}. \label{w2}
\end{eqnarray}
Due to invertibility of $q_a$ for all $a \in X$ we conclude from (\ref{w1}), (\ref{w2}) that 
\[h _{c \triangleright(b \triangleright a)} =h_{(c \triangleright b) \triangleright ( c\triangleright a)}\  
 \Rightarrow\ c \triangleright(b \triangleright a) = (c \triangleright b) \triangleright ( c\triangleright a).\]
 
We assume $c\triangleright a = c\triangleright b$, then $q_c h_{c\triangleright a} = q_c h_{c\triangleright b}$,  by the fourth relation in \ref{qualg}, we get $h_aq_c = h_bq_c$ and by the invertibility of $q_c$, $h_a = h_b,$ hence $a=b$, i.e. $a\triangleright$ is bijective and thus $(X, \triangleright)$ is a rack.

Moreover,  we recall for the quandle algebra $a\bullet b = b \bullet (b \triangleright a),$ then for $a=b$ and 
by recalling bijectivity and hence
left cancellativity in $(X, \bullet)$ we conclude that $a\triangleright a =a,$ i.e.  $(X, \triangleright )$ is a quandle.
\end{proof}

\noindent {\it Note:} In the case of a quandle algebra $\mathcal{A}$ (without assuming $X$ to be finite), if $\left(X,\bullet\right)$ is a quasigroup,  $c\triangleright$ is surjective. Indeed, if $b\in X$, and $r_c$ and $l_c$ denotes right and left multiplication by $c$, respectively, with respect to the binary operation $\bullet$, then there exists $x:= r^{-1}_cl_c\left(b\right)$ such that
$c\triangleright x = l^{-1}_c\left(x\bullet c\right) = b$. In this case, we also obtain injectivity without using relations  \eqref{qualg}. Indeed, if $x,y\in X$ are such that $c\triangleright x = c\triangleright y$, then $l^{-1}_cr_c\left(x\right) = l^{-1}_cr_c\left(y\right)$, hence $x=y$.
Finally, as in the last part of the proof of 
\cref{proro1}, we obtain that $\left(X, \triangleright\right)$ is a quandle.

\begin{lemma} \label{lemmac}
Let $c= \sum_{a\in X} h_a,$ then $c$ is a central element of the rack algebra ${\cal A}$. Also, $h_a^2 =h_a,$ for all $a \in X.$ \end{lemma}
\begin{proof} The proof is direct by means of the definition of the algebra ${\cal A}$ and Proposition  \ref{proro1}. Without loss of generality let $c=1_{\cal A},$  then it also 
follows that 
$h_a^2 =h_a,$ for all $a \in X.$ \end{proof}

\begin{defn}\label{homoa} 
Let ${\cal A}$ be a quandle algebra, then $q: X \to {\cal A},$ $a \mapsto q_a$ is called:
\begin{enumerate}
\item Strong $(X, \bullet)-$homomorphism  if  $q_a q_b = q_{a\bullet b},$ for all $a,b \in X$ 
\item Weak $(X, \bullet)-$homomorphism if there exists a map $F: X \to  {\cal A},$ $a \mapsto F_a,$ 
such that  $q_a q_b = F_{a\bullet b},$ for all $a,b \in X$
($F \neq q$; if $F=q,$ then one recovers 
a strong homomorphism).
\end{enumerate}
\end{defn}

\begin{exa} \label{exhom} 
We provide below the  basic cases examined in the previous sections.
\begin{enumerate}
\item Let $(X, \bullet)$ be a Yang-Baxter algebra for a solution given by a quandle 
$(X, \triangleright )$. Let also $(X, \bullet)$ be a group, then for all $a, b\in X,$  $b \triangleright a:= 
b^{-1}\bullet a\bullet  b.$ One can easily check that the map $q: X \to End(V),$  
$a\mapsto q_a := \sum_{x \in X} e_{x, a \triangleright x}$ is a strong group homomorphism, with 
$V:= \mathbb{C}X$ the free vector space on $X$ (see the notation after \cref{lemmagen}).
\item Let $(X, \bullet_s)$ be a Yang-Baxter algebra from \cref{prop1}, i.e. for all $a, b \in X$, 
$a\triangleright_s b := - f(a) + f(b) + a$ \ and \ $a\bullet_s b = f(a)+b$, for some bijective  $f$. 
Then $q: X \to End(V),$  $a\mapsto q_a := \displaystyle\sum_{x \in X} e_{x, a {\triangleright_s}\,x}$ is a weak  
algebra homomorphism.
\noindent Indeed, let us consider the map $F:X\to End(V),\,a\mapsto F_a:= \displaystyle\sum_{x\in X}e_{x,a\triangleright_s\left(f\left(x\right) + e\right)}$ with $e=f^{-1}\left(0\right)$. Then, if $f$ is a heap homomorphism, $q_aq_b = F_{a\bullet_s b}$, for all $a,b\in X$, and hence the map $q$ is a weak $\left(X,\bullet_s\right)-$homomorphism.
\item We have a similar example as in $(2)$ for the Yang-Baxter algebra $(X, \bullet_r)$ \cref{prop1}, i.e., for all $a, b \in X$, 
$a\triangleright_r b := a + f(b) - f(a)$ \ and \ $a\bullet_r b = b + f(a)$, for some bijective  $f$. 
Then $q: X \to End(V),$  $a\mapsto q_a := \displaystyle\sum_{x \in X} e_{x, a {\triangleright_r}\,x}$ is a weak  
algebra homomorphism. If we consider the map $F:X\to End(V),\, a\mapsto F_a:= \displaystyle\sum_{x\in X}e_{x,a\triangleright_r\left(e + f\left(x\right)\right)}$ with $e=f^{-1}\left(0\right)$, if $f$ is a heap homomorphism, then $q_aq_b = F_{a\bullet_r b}$, for all $a,b\in X$, namely, the map $q$ is a weak $\left(X,\bullet_r\right)-$homomorphism.
\end{enumerate}
\end{exa}

Having defined the rack algebra we are now 
in the position to identify the associated universal 
${\cal R}$-matrix (solution of the Yang-Baxter equation).
\begin{pro} 
Let ${\cal A}$ be the rack algebra  and ${\cal R} \in {\cal A} \otimes {\cal A}$ 
be an invertible element, such that ${\cal R} = \sum_{a\in X} h_a \otimes q_a.$
Then ${\cal R}$ satisfies the Yang-Baxter equation
\begin{equation}
{\cal R}_{12} {\cal R}_{13} {\cal R}_{23} = {\cal R}_{23} {\cal R}_{13} {\cal R}_{12},  \nonumber
\end{equation}
where ${\cal R}_{12} = \sum_{a\in X} h_a \otimes q_a \otimes 1_{\cal A}, $ ${\cal R}_{13} = \sum_{a\in X} h_a  
\otimes 1_{\cal A} \otimes q_a,$ and  ${\cal R}_{23} = \sum_{a\in X} 1_{\cal A} \otimes h_a \otimes q_a.$
\end{pro}
\begin{proof}
The proof is  a direct computation of the two sides of the Yang-Baxter equation 
(and use of the fundamental relations (\ref{qualg})):
\begin{eqnarray}
&& \mbox{LHS}: \quad   \sum_{a,b,c \in X} h_a h_b \otimes  q_a  h_c \otimes q_b q_c =  
\sum_{a,b,c \in X} h_a \otimes  q_a  h_c \otimes q_a q_c = \sum_{a,b,c \in X} h_a \otimes  
q_a  h_{a\triangleright c} \otimes q_a q_{a\triangleright c}\nonumber\\
&& \mbox{RHS}: \quad  \sum_{a,b,c\in X}   h_b h_a \otimes  h_c  q_a \otimes q_c q_b=  
\sum_{a,b,c\in X}    h_a \otimes   q_a h_{a\triangleright c} \otimes q_c q_a, \nonumber
\end{eqnarray}
where we have used that $a\triangleright$ is bijective. Then due to the basic relation $q_a q_b = q_b  q_{b \triangleright a},$ we show that LHS$=$RHS,
and this concludes our proof.
\end{proof}

\begin{rem} \label{remc}
The universal ${\cal R}-$matrix is obviously invertible, indeed from Lemma \ref{lemmac},
$\sum_{a\in X} h_a= 1_{\cal A},$ hence
${\cal R}^{-1} = \sum_{a\in X} h_a \otimes q_a^{-1}.$  
\end{rem}

\begin{rem} \label{remfu} 
{\bf Fundamental representation:} Let ${\cal A}$ be the rack algebra and $\rho: {\cal A} \to \mbox{End}(V)$ be the map defined by
\begin{equation}
q_a \mapsto \sum_{x \in X} e_{x, a \triangleright x}, \quad h_a\mapsto e_{a,a}. \label{remfu1}
\end{equation}
Then ${\cal R} \mapsto R:= 
\sum_{a,b\in X} e_{b,b} \otimes e_{a, b\triangleright a}. $
\end{rem}
 Let ${\cal P} = \sum_{a,b \in X} e_{a,b} \otimes e_{b,a}$ be the permutation (flip) operator, 
then the solution of the braid equation is the familiar rack solution,
 $r = {\cal P} R = \sum_{a,b} e_{a,b} \otimes e_{b, b \triangleright a}.$ 
We note that $R$ is invertible, because $a\triangleright: X \to X$ is a bijection from Definition \ref{def1},  then $R^{-1} = \sum_{a,b\in X} e_{b,b} \otimes e_{b\triangleright a, a}.$
\\
In general, from the universal $R-$matrix and the Yang-Baxter equation, and after recalling the representations of Remark
 \ref{remfu} and setting:\\
$(\rho \otimes \mbox{id}){\cal R} := L = \sum_{a\in X}e_{a,a} \otimes q_a,$ 
$~(\mbox{id} \otimes \rho){\cal R}:= \hat L = \sum_{a,b \in X} h_b \otimes e_{a, b \triangleright a},$ and \\ 
$(\rho \otimes \rho){\cal R} :=R = \sum_{a,b  \in X} e_{b, b} \otimes e_{a, b \triangleright a},$ 
consistent algebraic relations ensue:
\begin{eqnarray}
&& R_{12} L_{13} L_{23} = L_{23} L_{13} R_{12},  \label{a} \\
&& \hat L_{12} \hat L_{13} R_{23} = R_{23}  \hat L_{13} \hat L_{12}, \label{b} \\ 
&& L_{12} R_{13}
\hat L_{23} = \hat L_{23} R_{13} L_{12}, \label{c}
\end{eqnarray}
which lead to the rack algebra given in Definition (\ref{qualg}) and provide a consistency check on the algebraic relations (\ref{qualg}). 
This is the Faddeev-Reshetikhin-Takhtajan (FRT) construction \cite{FRT}.

\begin{rem} \label{rem} In general, 
let ${\cal R}=\sum_i f_i \otimes g_i \in {\cal C} \otimes {\cal C}$
(${\cal C}$ is the associated unital universal algebra) be a universal ${\cal R}$-matrix.
Let also,
$\rho: {\cal C} \to \mbox{End}(V), $ such that $(\rho \otimes \mbox{id}){\cal R} =: L$ 
and $(\rho \otimes \rho){\cal R} =: R,$ 
then the Yang-Baxter equation reduces to (\ref{a}).
Expression (\ref{a}) is equivalent to (\ref{funda}),
if a vector $v = \sum_a f_a \hat e_a \in V,$ $f_a \in \mathbb{C}$ exists, such that 
$ R v \otimes v  = v \otimes v. $
If the latter holds,  and if we set $L v =: \hat {\mathrm q},$ we conclude
\begin{eqnarray}
&&R_{12} L_{13}  L_{23}(v_1v_2 ) =L_{23} L_{13}  R_{12}(v_1 v_2)\ \Rightarrow\  R_{12} \hat {\mathrm q}_1 \hat {\mathrm q}_2 = 
\hat {\mathrm q}_2 \hat {\mathrm q}_1, \label{22}
\end{eqnarray}
where we recall the notation: $v_1= v \otimes \mbox{id} \otimes 1_{\cal C},$ $v_2 = \mbox{id} \otimes v \otimes1_{\cal C}.$
And if  $r := {\cal P} R,$ where ${\cal P}$ is the permutation operator,  then (\ref{22}) reads as, 
$~r_{12}\hat {\mathrm q}_{1}\hat  {\mathrm q}_2 = \hat {\mathrm q}_{1}\hat  {\mathrm q}_2.$
 
In the case for instance of non-degenerate,  invertible, set-theoretic (and quandle) solutions we can just consider 
$v:= \sum_{a \in X} \hat e_a.$
\end{rem}

\begin{exa}
Let ${\cal A}$ be the quandle algebra and $\rho: {\cal A} \to \mbox{End}(V)$ the fundamental representation 
$\pi: {\cal A} \to \mbox{End}(V),$ $q_a \mapsto {\mathrm q}_a:=  \sum_{b\in X} e_{b, a\triangleright b}, $ $h_a \mapsto e_{a,a}:$
\begin{eqnarray}
\sum_{x\in X} e_{x, a \triangleright x}  \sum_{x\in X} e_{y, b \triangleright y} =  \sum_{x\in X} e_{x, b \triangleright (a\triangleright x)}
\end{eqnarray}
\begin{enumerate}
\item If $(X, \bullet, e)$ is a group with identity $e$ and inverse denoted by $a^*,$ for every $a\in X, $ then $b \triangleright (a\triangleright x) = (a\bullet b) \triangleright x,$ for every $x\in X,$
and consequently $\rho_a \rho_b = \rho_{a\bullet b},$ i.e.  this is a strong group homomorphism.  In this case ${\mathrm q}_{a^*} = {\mathrm q}^{-1}_a$ and ${\mathrm q}_e ={\rho}(1_{\cal A}) =id_V.$
Notice also that ${\mathrm q}_a = {\mathrm q}_b \Rightarrow x \bullet (a\bullet b^*) = (a\bullet b^*) \bullet x, $ for every 
$x \in X,$ i.e.  $a\bullet b^*$  is a central element in the group $(X, \bullet).$
\item If $(X, \bullet_s)$ is the pre-Lie skew brace studied in subsection 3.3,  
with $a\triangleright_s b = -f(a) + f(b) +a,$ $~a \bullet_s b = f(a) +b,$ then 
$b \triangleright_s (a\triangleright_s x) =  (a\bullet_s b) \triangleright_s (x\bullet_s e),$ for every $x \in X$ 
(recall $e= f^{-1}(0)$) and consequently ${\mathrm q}_a {\mathrm q}_b = F_{a\bullet_s b}:= \sum_{x\in X} e_{x, (a\bullet_s b) \triangleright_s(x\bullet_s e)},$ 
i.e. this is a weak algebra homomorphism.  In this case ${\mathrm q}_{a}^T = {\mathrm q}^{-1}_a$  ($^T$ denotes transposition) and 
$\sum_{a\in x} e_{a,a,} = \rho(1_{\cal A}) = id_{V}.$ Notice also that if ${\mathrm q}_a = {\mathrm q}_b \Rightarrow a -b = -f(x)+f(a)- f(b) +f(x),$ for every  $x \in X.$
\end{enumerate}
\end{exa}

\begin{thm}\label{basica1}  
Let  ${\cal A}$ 
be the quandle algebra (Definition \ref{qualgd}), with $(X, \bullet, e)$ being a group. Let also
${\cal R}= \sum_{a\in X} h_a \otimes q_a$ be a solution of the Yang-Baxter equation and $q_a: X \to {\cal A}$ be a strong group homomorphism (i.e.  $q_a q_b = q_{a\bullet b}$). 

Then the structure 
$({\cal A}, \Delta, \epsilon, S,  {\cal R})$ is a quasi-triangular Hopf algebra:
\begin{itemize}
\item Co-product. $\Delta: {\cal A} \to {\cal A} \otimes {\cal A},$
$~\Delta(q_a^{\pm 1}) = q_a^{\pm 1}\otimes q_a^{\pm 1}$
and $\Delta(h_a) = \sum_{b,c \in X} 
h_b \otimes h_c\Big |_{b\bullet c = a}.$
\item Co-unit. $~\epsilon: {\cal A} \to k,$ $\epsilon(q_a^{\pm 1}) = 1,$ $~\epsilon(h_a) =  \delta_{a,e}.$ 
\item Antipode. $~S: {\cal A} \to {\cal A},$ $~S(q_a^{\pm 1}) =q_a^{\mp 1},$ $S(h_a)= h_{a^*},$ where $a^*$ 
is the inverse in $(X, \bullet)$ for all $a \in X.$
\end{itemize}
\end{thm}
\begin{proof} We recall that the maps 
$\Delta, \epsilon$ 
are algebra homomorphisms and $S$ is an anti-homomorphism,  consistent with the
relations of the quandle algebra.
We are now going to show all the axioms of a quasi-triangular Hopf algebra. 
We first identify,
\begin{eqnarray}
&& {\cal R}_{13} {\cal R}_{12} =\sum_{a\in X}  h_a \otimes q_a \otimes 
q_a=: \sum_{a\in X} h_a \otimes \Delta(q_a) = (\mbox{id} \otimes \Delta){\cal R},  \label{v1}\\
&& {\cal R}_{13} {\cal R}_{23} =\sum_{a,b\in X} h_a \otimes h_b \otimes q_c\Big |_{a\bullet b =c} 
=:  \sum_{c\in X} \Delta(h_c) \otimes  q_c  = (\Delta \otimes \mbox{id}){\cal R}, \label{v2}
\end{eqnarray}
hence  we immediately read of $\Delta(h_a),\ \Delta(q_a)$ as ($\Delta :{\cal A} \to {\cal A} \otimes {\cal A}$ 
should be an algebra homomorphism, but this can be explicitly checked via the distributivity condition $a\triangleright(b \bullet c) = (a \triangleright b) \bullet (a \triangleright c),$ which readily follows from, $a\triangleright b = a^{*}\bullet b \bullet a$):
\begin{equation}
\Delta(q_a) = q_a \otimes q_a,  \quad  \Delta(q^{-1}_a) = q^{-1}_a \otimes q^{-1}_a,
\quad \Delta(h_a)= \sum_{b,c \in X} h_b \otimes h_c\Big |_{b\bullet c =a}. \nonumber\end{equation}
Moreover, from the Yang-Baxter equation and relation (\ref{v1}), (\ref{v2}) we obtain
\begin{eqnarray}
  \Delta^{(op)}(q_a) {\cal R} = {\cal R} \Delta(q_a)\quad  \Delta^{(op)}(h_a) {\cal R} = {\cal R} \Delta(h_a),  \label{comm}
\end{eqnarray}
where $\Delta^{(op)} = \pi \circ \Delta,$ $\pi$ is the flip map.

Given the co-products of the generators we have to check co-associativity and 
also uniquely derive the counit $\epsilon: {\cal A} \to k$ (homomorphism) 
and antipode $S: {\cal A} \to {\cal A}$ (anti-homomorphism).
\begin{enumerate}[{(i)}]
\item Co-associativity.: $ (\mbox{id} \otimes \Delta)\Delta = (\Delta \otimes \mbox{id}).$
\begin{eqnarray}
&& (\mbox{id} \otimes \Delta)\Delta(q_a) =  (\Delta \otimes \mbox{id}) \Delta(q_a) =  
q_a \otimes q_a \otimes q_a, \quad  \nonumber\\
&& (\mbox{id} \otimes \Delta)\Delta(h_a) =  (\Delta \otimes \mbox{id}) \Delta(h_a) = 
\sum_{b,c,d\in X} h_b \otimes h_c \otimes h_d 
\Big |_{b\bullet c \bullet d =a}. \nonumber 
\end{eqnarray}
\item Counit: $ (\epsilon \otimes \mbox{id})\Delta(x) = (\mbox{id} \otimes \epsilon)
\Delta(x) =x,$ for all $x\in  \{q_a,\ q_a^{-1}, h_a\}.$\\
The generators $q_a$ are group-like elements, so $\epsilon(q_a) = 1,$ and 
\begin{equation}
\sum_{a,b\in X} \epsilon(h_a) h_b = \sum_{a,b} h_a \epsilon(h_b) 
\Big |_{a\bullet b =c} = h_c \Rightarrow \epsilon(h_a) = \delta_{a,e}.
\end{equation}
\item Antipode: $m\big ((S \otimes \mbox{id})\Delta(x)) =  
m\big ((\mbox{id} \otimes S )\Delta(x)) = \epsilon(x)1_{\cal A}$ for all $x\in \{q_a,\ q_a^{-1}, h_a\}.$\\
For $q_a,$ we immediately have  $S(q_a) = q_a^{-1}$ and (recall $h_{a} h_{b} = \delta_{a,b} h_a$ 
and $\sum_{a\in X} h_a = 1_{\cal A}$)
\begin{equation}
\sum_{a,b \in X} S(h_a) h_b \Big |_{a\bullet b =c} = \sum_{a,b \in X} h_a S(h_b) 
\Big |_{a\bullet b =c} = \delta_{c, e} 1_{\cal A}\ \Rightarrow\ S(h_a) = h_{a^{*}},  \nonumber\\
\end{equation}
where $a^*$ is the inverse in $(X, \bullet)$ for all $a \in X.$

We conclude that $({\cal A}, \Delta, \epsilon, S, {\cal R})$ is a quasi-triangular Hopf algebra.
\hfill \qedhere
\end{enumerate}
\end{proof}

\begin{cor}
Let ${\cal A}$ be the rack algebra, then the subalgebra  consisting of the indeterminates $1_{\cal A},\ q_a\ q_a^{-1},$ and relations for all $a\in X$ 
\begin{eqnarray}
q_a^{-1} q_a = q_a q_a^{-1} = 1_{\cal A},\quad q_a q_b = q_b  q_{b \triangleright a},  \label{qualg-}
\end{eqnarray}
is a Hopf algebra with:
\begin{enumerate}
    \item Co-product. $\Delta: {\cal A} \to {\cal A} \otimes {\cal A},$
$~\Delta(q_a^{\pm 1}) = q_a^{\pm 1}\otimes q_a^{\pm 1}.$
\item Co-unit. $\epsilon: {\cal A} \to k,$ $~\epsilon(q_a^{\pm 1})=1.$ 
\item Antipode. $S: {\cal A} \to {\cal A},$ $~S(q_a^{\pm 1}) =q_a^{\mp 1}.$ 
\end{enumerate}
\end{cor}
\begin{proof}
The proof is straightforward by Theorem \ref{basica1}.
    \end{proof}

\begin{rem} \label{basica2} 
Let $(X, \triangleright)$ be a quandle and $(X, \bullet)$ be a magma with a left neutral element,
such that $a\bullet b = b \bullet(b \triangleright a).$ Let also ${\cal A}$ 
be the quandle algebra (Definition \ref{qualgd}),
${\cal R}= \sum_{a\in X} h_a \otimes q_a$ is a solution of the Yang-Baxter equation 
and $q_a: X \to {\cal A}$ is a weak algebra homomorphism (Definition \ref{homoa}) and $a\triangleright(b \bullet c) = (a \triangleright b) \bullet (a \triangleright c).$ Then the structure 
$({\cal A}, \Delta, \epsilon,  {\cal R})$ is a quasi-triangular quasi-bialgebra.

Indeed,
we first recall that ${\cal R}$ is a solution of the Yang-Baxter equation, thus we identify 
(recall we assumed $q_a q_b = F_{a\bullet b}$, see Definition \ref{homoa})
\begin{eqnarray}
&& {\cal R}_{13} {\cal R}_{12} =\sum_{a\in X} h_a \otimes q_a \otimes q_a=: 
 \sum_{a}h_a \otimes \Delta(q_a) = (\mbox{id} \otimes \Delta){\cal R}, \label{costr0a} \\
&& {\cal R}_{13} {\cal R}_{23} =\sum_{a,b\in X} h_a \otimes h_b \otimes F_c\Big |_{a\bullet b =c} 
= : \sum_{c\in X} \Delta(h_c) \otimes  F_c \neq (\Delta\otimes \mbox{id}){\cal R}, ~ (F\neq q).  \label{costra}
\end{eqnarray}
Hence  we immediately read of $\Delta(h_a),\ \Delta(q_a)$ as:
\begin{equation}
\Delta(q_a) = q_a \otimes q_a,  \quad  \Delta(h_a)= \sum_{b,c \in X} h_b \otimes h_c\Big |_{b\bullet c =a}.\nonumber
\end{equation}
From the Yang-Baxter equation and (\ref{costr0a}), (\ref{costra}) 
we conclude
$~\Delta^{(op)}(q_a) {\cal R} =  \Delta(q_a)  {\cal R},$ ~~$\Delta^{(op)}(h_a) {\cal R} = {\cal R} \Delta(h_a),$
where $\Delta^{(op)} = \pi \circ \Delta,$ $\pi$ is the flip map.

Given the co-products of the generators we have to check co-associativity and 
also uniquely derive the counit $\epsilon: {\cal A} \to k$ (homomorphism) 
and antipode $S: {\cal A} \to {\cal A}$ (anti-homomorphism).
\begin{enumerate}[{(i)}]
\item Co-associativity.: $ (\mbox{id} \otimes \Delta)\Delta = (\Delta \otimes \mbox{id})\Delta.$ Coassociativity for 
$\Delta(q_a^{\pm 1})$ is obvious, but
\begin{eqnarray}
 ~~~(\Delta \otimes \mbox{id}) \Delta(h_a) = 
\sum_{b,c,d\in X} h_b \otimes h_c \otimes h_d 
\Big |_{(b\bullet c) \bullet d =a}, ~~ (\mbox{id} \otimes \Delta)\Delta(h_a) = \sum_{b,c,d\in X} h_b \otimes h_c \otimes h_d 
\Big |_{b\bullet (c \bullet d )=a}. \nonumber
\end{eqnarray}
Notice that expressions above are in general distinct, given that the operation $\bullet$
 is not necessarily associative.

\item Counit: $ (\epsilon \otimes \mbox{id})\Delta(x) = (\mbox{id} \otimes \epsilon)\Delta(x) =x,$ for all $x\in \{q_a,\ q_a^{-1}, h_a\}.$\\
The generators $q_a$ are group-like elements, so $\epsilon(q_a) = 1.$ For the generators $h_a,$ the following must be computed:
$\sum_{a,b\in X} \epsilon(h_a) h_b\Big |_{a\bullet b =c}$ and $ \quad  \sum_{a,b} h_a \epsilon(h_b) 
\Big |_{a\bullet b =c}.$
Note that if a left neutral element $e$ exists and if $\epsilon(h_a)= \delta_{a,e},$ then
$(\epsilon \otimes \mbox{id})\Delta(h_c)= \sum_{a,b\in X} \epsilon(h_a) h_b \Big |_{a\bullet b = c} = h_c,$
which suggests that in this case we are dealing with a quasi-bialgebra.
\end{enumerate}
Moreover,  because of (\ref{costr0a}), (\ref{costra}) and the lack of coassociativity of $\Delta(h_a)$ we conclude 
that $({\cal A}, \Delta, \epsilon)$ is not a bialgebra, 
but rather a quasi-bialgebra,  which needs to be fully identified by deriving the coassociator 
(see also relevant study for the set-theoretic quasi-bialgebra \cite{Doikou1, DoGhVl}). 
This analysis however will be presented in detail in a future work.
\end{rem}

\subsection{The set-theoretic algebras}

\noindent In this subsection we suitably extend the rack and quandle algebras in order to construct 
the universal ${\cal R}-$matrix associated to general set-theoretic solutions of the Yang-Baxter equation.

We first define the {\it decorated rack} algebra and the set-theoretic Yang-Baxter algebra (see also \cite{braceh}).

\begin{defn}  \label{setalgd1} (The decorated rack algebra.) Let ${\cal A}$ 
be the rack algebra (Definition \ref{qualg}). Let also $\sigma_a, \ \tau_b: X\to X,$ and $\sigma_a$ be  bijective for all $a\in X$. We say that the unital, associative algebra $\hat {\cal A}$ over $k,$
generated by indeterminates  $1_{\hat {\cal A}} ,q_a, q_a^{-1}, h_a, \in {\cal A}$ ($h_a = h_b \Leftrightarrow a =b$) and 
$w_a, w^{-1}_a\in \hat {\cal A},$ $a \in X,$ 
and relations, $a,b \in X:$
\begin{eqnarray}
&& q_a^{-1} q_a = q_aq_a^{-1} =1_{\hat {\cal A}}, ~~ q_a q_b = q_b 
q_{b \triangleright a}, ~~  h_a  h_b =\delta_{a,b} h_a, ~~ q_b  h_{b\triangleright a} = h_a q_b,\nonumber\\
&&w_a^{-1} w_a =w_aw_a^{-1} =1_{\hat {\cal A}}, ~~  w_a w_b = 
w_{\sigma_a(b)} w_{\tau_{b}(a)} ~~  w_a h_b = h_{\sigma_a(b)} w_a, 
~~ w_a q_b = q_{\sigma_a(b)} w_a \label{qualgbb}
\end{eqnarray}
is a {\it decorated rack algebra}.
\end{defn}

\begin{lemma} \label{lemmad}
Let $c= \sum_{a\in X} h_a,$ then $c$ is a central element of the decorated rack 
 algebra $\hat {\cal A}$.
\end{lemma}
\begin{proof} The proof is straightforward by means of the definition of the algebra $\hat {\cal A}.$ 
\end{proof}
We consider henceforth, without loss of generality,  $c=1_{\hat {\cal A}}$ (see also Lemma \ref{lemmac}). 

\begin{pro} 
\label{qua2} 
Let $\hat {\cal A}$ be the decorated rack algebra,
then for all $a,b,c \in X,$
\begin{eqnarray}
&&  \sigma_{a}(\sigma_b(c))= \sigma_{\sigma_{a}(b)}(\sigma_{\tau_{b}(a)}(c))  \quad \& \quad   
\sigma_c(b) \triangleright \sigma_{c}(a) = \sigma_c(b \triangleright a). \nonumber
\end{eqnarray}
\end{pro}
\begin{proof}
We compute $w_a w_b h_c$ using the associativity of the algebra,
\begin{eqnarray}
&& w_a w_b h_c =  w_{\sigma_a(b)} w_{\tau_b(a)} h_c =
h_{\sigma_{\sigma_{a}(b)}(\sigma_{\tau_{b}(a)}(c))}w_{\sigma_a(b)} w_{\tau_b(a)},  \nonumber \label{wb1} \\
&& w_a w_b h_c =h_{\sigma_a(\sigma_b(c))} w_a w_b  = h_{\sigma_a(\sigma_b(c))}w_{\sigma_a(b)} w_{\tau_{b}(a)}. \label{wb2} \nonumber
\end{eqnarray}
From the equations above and the invertibility of $w_a$ for all $a\in X$ we conclude 
\begin{equation}
h_{\sigma_{\sigma_{a}(b)}(\sigma_{\tau_{b}(a)}(c))} = h_{\sigma_a(\sigma_b(c))}\  
\Rightarrow \ \sigma_{\sigma_{a}(b)}(\sigma_{\tau_{b}(a)}(c))= \sigma_a(\sigma_b(c)).  \label{basiko} \nonumber
\end{equation}
\\
We also compute $h_aq_b w_c$:
\begin{eqnarray}
&& h_a q_b w_c = h_a w_c q_{\sigma_c^{-1}(b)} = w_c h_{\sigma_{c}^{-1}(a)}q_{\sigma_c^{-1}(b)} =  w_c q_{\sigma_c^{-1}(b)} 
h_{\sigma_c^{-1}(b)\triangleright \sigma_{c}^{-1}(a)} \nonumber\\
&&  h_aq_b w_c = q_b h_{b \triangleright a}w_c = q_b w_c h_{\sigma^{-1}_c(b \triangleright a)} = 
w_c q_{\sigma_c^{-1}(b)} h_{\sigma^{-1}_c(b \triangleright a)}. \nonumber
\end{eqnarray}
From the equations above and the invertibility of $q_a,\ w_a$ for all $a\in X$:
\begin{equation}
h_{\sigma_c^{-1}(b)\triangleright \sigma_{c}^{-1}(a)} = h_{\sigma^{-1}_c(b \triangleright a)}\ \Rightarrow \  
\sigma_c^{-1}(b)\triangleright \sigma_{c}^{-1}(a) = \sigma^{-1}_c(b \triangleright a),\nonumber
\end{equation}
from the latter it immediately follows, $ \sigma_c(b)\triangleright \sigma_{c}(a) = \sigma_c(b \triangleright a).$
\end{proof}
It is worth pointing out the compatibility of the above proposition with Definition \ref{def:twist:shelf}, which indicates the consistency of our construction.

\begin{defn}  
\label{setalgd} 
(The set-theoretic  Yang-Baxter algebra.) 
Let ${\cal A}$ 
be the quandle algebra. Let also $\sigma_a, \ \tau_b: X\to X,$ and $\sigma_a$ be bijective for all $a\in X.$  We say that the unital, associative algebra $\hat {\cal A}$ over $k,$
generated by indeterminates  $1_{\hat {\cal A}} ,q_a, q_a^{-1}, h_a, \in {\cal A}$ ($h_a = h_b \Leftrightarrow a =b$) and 
$w_a, w^{-1}_a\in \hat {\cal A},$ $a \in X,$ 
and relations, (\ref{qualgbb}) is a set-theoretic  Yang-Baxter algebra.
\end{defn}

\begin{pro} \label{basica2b} (Hopf algebra) 
Let $\hat {\cal A}$ 
be the set-theoretic Yang-Baxter algebra,
 ${\cal R} = \sum_{b\in X} h_b\otimes q_b$ 
be the rack universal ${\cal R}$-matrix,  and $({\cal A}, \Delta,  \epsilon, S,  {\cal R})$  
be the quasi-triangular Hopf algebra of Theorem \ref{basica1}.  Let also for all $a,b,x \in X,$ 
\begin{equation}
 \sigma_x(a) \bullet \sigma_x(b) = \sigma_x(a\bullet b). \label{condition0}
\end{equation} 
Then, 
\begin{enumerate}
\item $(\hat {\cal A}, \Delta, \epsilon, S)$ is a Hopf algebra with $\Delta(w_a) = w_a\otimes w_a,$ 
for all $a \in X.$
\item  $\Delta(w_a) {\cal R} = {\cal R} \Delta(w_a),$ for all $a \in X.$
\end{enumerate}
\end{pro}
\begin{proof}  
In our proof below we are using the Definition \ref{setalgd} and (\ref{condition0}).
\begin{enumerate}
\item The coproduct $\Delta$ is an algebra homomorphism.  It is sufficient to check below the 
consistency of all algebraic relations of Definition \ref{setalgd} for the corresponding 
coproducts and (\ref{condition0}).
Then we have for $Y_b \in \{h_b,\ q_b\}$ and for all $a\in X,$
\begin{equation}
\Delta(w_a) \Delta(w_b) = \Delta(w_{\sigma_a(b)}) \Delta(w_{\tau_b(a)}),  \quad   
\Delta(w_a) \Delta(Y_b) = \Delta(Y_{\sigma_a(b)}) \Delta(w_a). \nonumber
\end{equation}
Also, $w_a$ is a group-like element, thus the counit and antipode are given as: 
$\epsilon(w_a) =1$ and $S(w_a) = w^{-1}_a.$ Recall that the coproducts, counits 
and antipodes of the generators $h_a,\ q_a$ are given in
Theorem \ref{basica1}.

\item By a direct computation and using the algebraic relations of the 
Definition \ref{setalgd} we conclude for all $a \in X,$
$~\Delta(w_a) {\cal R} = {\cal R} \Delta(w_a).$
\hfill \qedhere
\end{enumerate}
\end{proof}

\begin{cor}\label{cor:basica2b}
Let $\hat {\cal A}$ be the decorated rack algebra, then the subalgebra  consisting of the indeterminates $1_{\hat{\cal A}},\ q_a^{\pm 1}\ w_a^{\pm 1},$ and relations for all $a\in X$ 
\begin{eqnarray}
&& q_a^{-1} q_a = q_a q_a^{-1} = 1_{\cal A},~~ q_a q_b = q_b  q_{b \triangleright a}, \nonumber \\
&& w_a^{-1} w_a = w_a w^{-1}_a = 1_{\hat {\cal A}} 
 ~~ w_a w_b = w_{\sigma_a(b)} w_{\tau_b(a)}, ~~ w_a q_b = q_{\sigma_a(b)} w_a  \label{qualg-b}
\end{eqnarray}
is a Hopf algebra with:
\begin{enumerate}
    \item Co-product. $\Delta: {\cal A} \to {\cal A} \otimes {\cal A},$ 
$~\Delta(q_a^{\pm 1})= 
q_a^{\pm 1}\otimes q_a^{\pm 1},$  $~\Delta(w_a^{\pm 1}) = 
w_a^{\pm 1}\otimes w_a^{\pm 1}.$
\item Co-unit. $\epsilon: {\cal A} \to k,$ $~\epsilon(q_a^{\pm 1})= 1,$  $~\epsilon(w_a^{\pm 1}) =1.$
\item Antipode. $S: {\cal A} \to {\cal A},$ $~S(q_a^{\pm 1}) =q_a^{\mp 1}.$ 
\end{enumerate}
\end{cor}
\begin{proof}
The proof is straightforward by Theorem \ref{basica1} and Proposition \ref{basica2b}.
\end{proof}

\begin{rem} Proposition \ref{basica2b} can be generalized for $(X, \bullet)$ 
being a magma , if for all $x\in X,$
$\exists g_x: X\to X,$ $a\mapsto g_x(a),$ being a bijection,  
such that for all $a,b,x \in X,$  there exists $e \in X,$ $\sigma_e(b) = b=g_e(b)$ and
\begin{equation}
 \sigma_x(a) \bullet \sigma_x(b) = g_x (a\bullet b). \label{condition0b}
\end{equation} 
Note that in this case due to condition (\ref{condition0b}) we conclude that
\begin{eqnarray}
   \Delta(w_a) \Delta(h_b) =  \Delta(h_{g_a(b)}) \Delta(w_a),    
   \label{weak2}
\end{eqnarray}
that is in this case $\Delta: X \to X \otimes X$ is not an $\hat {\cal A}$ algebra homomorphism anymore.
\end{rem}

We now introduce some handy notation that can be used in the following:  
let $i,j,k \in \{1,2,3\},$ then ${\cal F}_{jik} = \pi_{ij} \circ {\cal F}_{ijk}$ and 
${\cal F}_{ikj} = \pi_{jk} \circ {\cal F}_{ijk},$ where  $\pi$ is the flip map. 

\begin{thm} \label{twist2} (Drinfel'd twist \cite{Drinfel'd}) 
Let ${\cal R} = \sum_{a\in X} h_a \otimes q_a\in {\cal A} \otimes {\cal A}$ 
be the rack universal ${\cal R}-$matrix. Let also $\hat {\cal A}$
be the decorated rack algebra and ${\cal F}\in \hat {\cal A} \otimes \hat  {\cal A},$  
such that ${\cal F} =\sum_{b\in X}h_b \otimes w_b^{-1},$
${\cal R}_{ij}^F := {\cal F}_{ji} {\cal R}_{ij} {\cal F}_{ij}^{-1},$ $i,j \in \{1,2,3\}.$
We also define:
\begin{eqnarray}
{\cal F}_{1,23} := \sum_{a\in X} h_a\otimes w_a^{-1} \otimes 
w_a^{-1}=,  \quad {\cal F}^*_{12,3} : =  \sum_{a,b \in X}h_a\otimes h_{\sigma_a(b)} 
\otimes w^{-1}_{b} w^{-1}_a.
\end{eqnarray}
Let also  for every
$a,b \in X,$ $~b\triangleright a = \sigma_b(\tau_{\sigma^{-1}_a(b)}(a)).$ 
Then, the following statements are true:
\begin{enumerate}
\item ${\cal F}_{12} {\cal F}^*_{12,3} ={\cal F}_{23} {\cal F}_{1,23} =: {\cal F}_{123}.$
\item For  $i,j,k \in \{1,2,3\}$:
(i) ${\cal F}_{ikj} {\cal R}_{jk} ={\cal R}_{jk}^F {\cal F}_{ijk}$
and 
(ii)  ${\cal F}_{jik} {\cal R}_{ij} ={\cal R}_{ij}^F {\cal F}_{ijk}.$
\end{enumerate}
That is,  ${\cal F}$ is an  admissible Drinfel'd twist.
\end{thm}
\begin{proof}
The proof is straightforward based on the underlying algebra $\hat {\cal A}.$
\begin{enumerate}
\item Indeed, this is proved by a direct computation and use of the decorated rack algebra. 
In fact, ${\cal F}_{123} = \sum_{a,b\in X} h_a \otimes h_b w_a^{-1} \otimes w_b^{-1} w_a^{-1}.$
\item Given the notation introduced before the theorem it  suffices to show  that
${\cal F}_{132} {\cal R}_{23} = {\cal R}^F_{23} {\cal F}_{123}$ and ${\cal F}_{213} 
{\cal R}_{12} = {\cal R}^F_{12} {\cal F}_{123}.$ \\
\noindent (i) 
Due to the fact that for all $a \in X,$
$ \Delta(w_a){\cal R} =  {\cal R} \Delta(w_a)$ 
(see Proposition \ref{basica2b})
we arrive at ${\cal F}_{1,32} {\cal R}_{23} = {\cal R}_{23} F_{1,23},$ then
\[{\cal F}_{132} {\cal R}_{23} = {\cal F}_{32} {\cal F}_{1,32} {\cal R}_{23} ={\cal F}_{32} 
{\cal R}_{23} {\cal F}_{1,23} = {\cal R}_{23}^F{\cal F}_{123}. \]
\noindent (ii) By means of the relations 
of the decorated rack algebra $\hat {\cal A}$ we compute:
\[ {\cal F}^*_{21,3}{\cal R}_{12} = \sum_{a,c\in X} h_a \otimes q_a h_{a\triangleright c} \otimes 
 (w_c  w_{\sigma^{-1}_c(a)} )^{-1},\quad {\cal R}_{12}{ \cal F}^*_{12,3}= \sum_{a,b \in X} 
h_a \otimes q_a h _{\sigma_a(b)} \otimes (w_aw_b)^{-1}.\] 
Due to the fact that $~b\triangleright a = \sigma_b(\tau_{\sigma^{-1}_a(b)}(a))$ and $w_a w_b = w_{\sigma_a(b)} w_{ \tau_b(a)}$
we conclude that ${\cal F}^*_{21,3}{\cal R}_{12} = 
{\cal R}_{12}{ \cal F}^*_{12,3}$ and consequently (recall ${\cal F}_{213} = {\cal F}_{21} {\cal F}^*_{21,3}$)
\[{\cal F}_{213} {\cal R}_{12} = {\cal F}_{21} {\cal F}^*_{21,3} {\cal R}_{12} 
={\cal F}_{21} {\cal R} _{12}{\cal F}^*_{12,3} = {\cal R}_{12} ^F{\cal F}_{123}. \hfill \qedhere
\]
\end{enumerate}
\end{proof}

\begin{cor} \cite{Drinfel'd} Let ${\cal F}$ be an admissible twist and ${\cal R}$ be a solution of the Yang-Baxter equation. 
Then ${\cal R}^F:= {\cal F}^{(op)} {\cal R} {\cal F}^{-1}$  (${\cal F}^{(op)} =\pi \circ {\cal F},$ $\pi$ is the flip map) 
is also a solution of the Yang-Baxter equation.
\end{cor}
\begin{proof}
The proof is quite straightforward,  \cite{Drinfel'd} (see also proof in \cite{Doikou1} for set-theoretic solutions),  
we just give a brief outline here: if ${\cal F}$ is admissible,  then from the YBE and due to Theorem \ref{twist2}:
\begin{equation}
{\cal  F}_{321} {\cal R}_{12} {\cal R}_{13} {\cal R}_{23} ={\cal  F}_{321} 
{\cal R}_{23}{\cal R}_{13}{\cal R}_{12}\ \Rightarrow\  {\cal R}^F_{12} {\cal R}^F_{13} {\cal R}^F_{23}{\cal  F}_{123} 
=  {\cal R}^F_{23}{\cal R}^F_{13}{\cal R}^F_{12}{\cal  F}_{123}. \nonumber
\end{equation}
But ${\cal F}_{123}$ is invertible, hence ${\cal R}^F$ indeed satisfies the YBE.
\end{proof}
We note the direct correspondence of the above Proposition and Corollary with Theorem \ref{le:lndsol}.

We are now going to examine the twisted ${\cal R}-$matrix as well 
as the twisted co-products (see also \cite{Doikou1, DoGhVl, DoiRyb22}). Henceforth, $\hat {\cal A}$ denotes the set-theoretic Yang-Baxter algebra.

\begin{rem} 
(Twisted universal ${\cal R}$-matrix) It is worth extracting the explicit 
expressions of the twisted universal ${\cal R}-$matrix and the twisted coproducts of the algebra. 
We recall the admissible twist ${\cal F} = \sum_{b\in X}h_b \otimes w_b^{-1}.$
\begin{itemize}

\item The twisted ${\cal R}-$matrix: 
\[{\cal R}^F = {\cal F}^{(op)} {\cal R} {\cal F}^{-1} = \sum_{a,b\in X}h_bw_a^{-1}  
 \otimes h_aq_{\sigma_a(b)} w_{\sigma_a(b)}.\]

\item The twisted coproduts: $\Delta_F(y) = {\cal F} \Delta(y) {\cal F^{-1}},$ $y\in \hat {\cal A}$ 
and we recall, for $a\in X,$
\[ \Delta(w_a) =w_a\otimes w_a,\quad \Delta(h_a) = 
\sum_{b,c\in X} h_b \otimes h_c\Big |_{b\bullet c =a}, \quad \Delta(q_a) =q_a \otimes q_a.\]
Then, the twisted coproducts read as:
$~\Delta_F(w_a) = \sum_{b \in X} h_{\sigma_a(b)} w_a \otimes w _{\tau_b(a)},$\\
$\Delta_F(h_a) = \sum_{b \in X} h_b \otimes w_b^{-1} h_c w_b\Big |_{b\bullet c =a},$ $~\Delta_F(q_a) =\sum_{b \in X} q_a h_{a\triangleright b} \otimes  w_b^{-1} q_a 
w_{a\triangleright b},$
\end{itemize}
and it immediately follows that ${\cal R}^F \Delta_F(Y) = \Delta_F^{(op)}(Y) {\cal R}^F,$ $Y \in \hat {\cal A}.$

Note that even if $(X, \bullet)$ is a group coassociativity for the above twisted coproducts does not hold any more,  that is 
$(\hat {\cal A}, \Delta_F, \epsilon, S )$ is not a Hopf algebra, but it is rather a quasi Hopf-bialgebra (see also 
\cite{DoGhVl, DoiRyb22} for a more detailed discussion on the issue). 
The identification of the universal associator of the quasi-Hopf algebra will be presented elsewhere.
\end{rem}

\begin{rem} \label{remfu2b}  
{\bf Fundamental representation $\&$ the set-theoretic solution:}\\ 
Let $\rho: \hat  {\cal A} \to \mbox{End}(V),$ such that
\begin{equation}
q_a \mapsto \sum_{x \in X} e_{x, a \triangleright x}, \quad h_a\mapsto e_{a,a}, \quad  
w_a \mapsto \sum_{b \in X} e_{\sigma_a(b),b},\label{repbb}
\end{equation}
then $ {\cal R}^F \mapsto R^F:= 
\sum_{a,b\in X} e_{b,\sigma_a(b)} \otimes e_{a, \tau_b(a)},$
where we recall that $\tau_{b}(a):=\sigma_{\sigma_a(b)}^{-1}(\sigma_a(b) 
\triangleright a)$ (see also \cite{Sol, LebVen, DoiRyb22}).
\end{rem}
Let ${\cal P} = \sum_{a,b \in X} e_{a,b} \otimes e_{b,a}$ be the permutation (flip) operator, 
then the solution of the braid equation is the familiar set-theoretic solution,
 $r^F = {\cal P} R^F = \sum_{a,b\in X} e_{a,\sigma_a(b)} \otimes e_{b, \tau_b(a)}.$
In general, from the universal ${\cal R}$-matrix and the Yang-Baxter equation, 
and after recalling the representations 
(\ref{repbb}) and setting ${\mathrm f}_{b,a}:=h_b w_a^{-1},$ ${\mathrm g}_{a,b} := h_aq_{\sigma_a(b)} w_{\sigma_a(b)}$:\\
$(\rho \otimes \mbox{id}){\cal R}^F := L^F = \sum_{a,b \in X}e_{b, \sigma_a(b)} \otimes {\mathrm g}_{a,b},$ 
$~(\mbox{id} \otimes \rho){\cal R}^F:= \hat L^F = \sum_{a,b \in X} {\mathrm f}_{b,a} \otimes e_{a, \tau_b(a)},$ and \\ 
$(\rho \otimes \rho){\cal R}^F :=R^F = \sum_{a,b  \in X} e_{b,\sigma_a(b)} \otimes e_{a, \tau_b(a)},$ 
the consistent algebraic relations (\ref{a})-(\ref{c}) are satisfied (the interested reader is also 
referred to \cite{DoiRyb22} for detailed computations). 
These  lead to the set-theoretic Yang-Baxter algebra and provide a 
consistency check on the associated algebraic relations.

\begin{exa}
We notice for the fundamental representation 
$\rho: \hat {\cal A} \to \mbox{End}(V),$ $q_a \mapsto {\mathrm q}_a:=  \sum_{b\in X} e_{b, a\triangleright b}, $ 
$h_a \mapsto e_{a,a},$ $w_a \mapsto W_a:= \sum_{b \in x} e_{\sigma_a(b) ,b}$ 
\begin{eqnarray}
W_a W_b =  \sum_{c\in X} e_{\sigma_a(\sigma_b(c)),c.}
\end{eqnarray}
\begin{enumerate}
\item Let $(X,+,  \circ)$ be a skew brace and $\sigma_a(b) = -a +a\circ b.$  In this case $W_{a^{-1}} = W^{-1}_a,$ 
$W_{1} = id_V$ and $W_a W_b = W_{a\circ b}= W_{\sigma_a(b)} W_{\tau_b(a)},$  i.e.  $W$ is a strong 
$(X,\circ)-$homomorhism.  Notice also that if $X$ is two-sided brace,  $W_a = W_b \Rightarrow \sigma_a(c) = \sigma_b(c)\Rightarrow (a-b) \circ c^{-1} = a-b+c^{-1},$ 
$\forall c \in X,$ i.e. $a-b\in Socle(X).$ 
We also note that
if  $\tau_{b}(a):=\sigma_{\sigma_a(b)}^{-1}(\sigma_a(b) 
\triangleright a)$ and $a\bullet b = a+b,$ $a\triangleright  b = -a +b +a,$ then $\tau_b(a) = (\sigma_a(b))^{-1}\circ a \circ b.$

\item Let $(X,+,  \circ)$ be a skew brace and $\sigma^z_a(b) = -a\circ z +a\circ b\circ z.$  
In this case $W_a^T = W^{-1}_a$ ($^T$ denotes transposition), $\sum_{a\in X} e_{a,a}= id_V$ 
and $W_a W_b = G_{a\circ b} : =\sum_{c \in X} e_{\sigma^{z\circ z}_{a\circ b}(c),c} = W_{\sigma_a(b)} W_{\tau_b(a)},$ i.e.  
$W$ is a weak $(X,\circ)-$homomorphism.  Notice also that if $X$ is a two-sided brace,
$W_a = W_b \Rightarrow \sigma_a(c) = \sigma_b(c)\Rightarrow (a-b) \circ c = a-b+c,$ 
$\forall c \in X,$ i.e $a-b\in Socle(X).$ 
\end{enumerate}
\end{exa}

\begin{pro} \label{rr1} 
Let $(X, +, \circ)$ be a skew brace,  $\sigma^f_a, \tau^f_b: X \to X,$ such that $a\circ b = \sigma^f_a(b) \circ \tau^f_b(a)$ and $\sigma^f_a$ is a bijection.  Let also
$(X, \bullet)$ be a magma,  such that $a\bullet b = b \bullet (b \triangleright a),$
and for all 
$a,b \in X,$ $a\circ(\sigma_a^f)^{-1}(b) = a \bullet b.$
We also assume that a bijection $f: X \to X,$ such that for all 
$a,b,c \in X,$ $f(a-b+c) = f(a)-f(b) +f(c)$ exists:

\begin{enumerate}[{(i)}]
\item If (1) $~a \bullet_s b =  f(a) +b,$ (2) $~b\triangleright_s  a=-f(b) +f(a) +b$ and (3) $f(a\circ b) =a \circ f(b) - a + f(a),$ then, for all $a,b \in X,$
 \begin{enumerate}[{a.}]
\item $~\sigma^f_a(b) = - f(a) +a\circ b.$ 
\item $~b\triangleright_s a = \sigma^f_b(\tau^f_{(\sigma^{f})^{-1}_a(b)}(a)).$
\item $~\sigma^f_a(b) \triangleright_s \sigma^f_a(c)= \sigma^f_a(b \triangleright_s c).$ 
\item $~\sigma^f_a(\sigma^f_b(c)) = \sigma^f_{a\circ b}(c),$ $~\tau^f_c(\tau_b^f(a)) =\tau^{f^2}_{b\circ c}(a).$
\item $~\sigma^f_a(b) \bullet_s \sigma^f_a(c) =\sigma^{f^2}_a(b \bullet_s c).$
\end{enumerate}

\item If (1) $~a \bullet_r b = b+ f(a), $ (2) $~b\triangleright_r  a= b +f(a) -f(b)$ and (3) $~f(a\circ b) =f(a) -a+ a \circ f(b),$  then, for all $a,b \in X,$
\begin{enumerate} [{a.}]
\item $~\sigma^f_a(b) = a\circ b -f(a).$ 
\item $~b\triangleright_r a = \sigma^f_b(\tau^f_{(\sigma^{f})^{-1}_a(b)}(a)).$
\item $~\sigma^f_a(b) \triangleright_r \sigma^f_a(c)= \sigma^f_a(b \triangleright_r c).$
\item $~\sigma^f_a(\sigma^f_b(c))= \sigma^f_{a\circ b}(c),$ $~\tau^f_c(\tau_b^f(a)) =\tau^{f^2}_{b\circ c}(a).$
\item $~\sigma^f_a(b) \bullet_r \sigma^f_a(c) = \sigma^{f^2}_a(b \bullet_r c).$ 
\end{enumerate}

\end{enumerate}
\end{pro}
\begin{proof} We first notice from condition (3) for $a=b=1$ that $f(1) = 0$ (recall in every skew brace $0=1$). 
We only prove part  \it{(i)}:
\begin{enumerate}[{a.}]
\item $~a\circ (\sigma_a^f)^{-1}(b) = a \bullet b\ \Rightarrow\ (\sigma_a^f)^{-1}(b) = 
a^{-1}\circ (a\bullet b)\ \Rightarrow\ (\sigma_a^f)^{-1}(b) = a ^{-1}\circ (f(a) +b).$
But,
\[(\sigma_a^{f})^{-1}(\sigma^f_a(b)) = b\ \Rightarrow\  a^{-1} \circ (f(a)+ \sigma^f_a(b)) = 
b\ \Rightarrow\  \sigma^f_a(b) = -f(a) + a\circ b. \]
\item $~\sigma^f_{b}(\tau^f_{(\sigma_a^{f})^{-1}(b)}(a)) = \sigma^f_b(b^{-1}\circ a\circ(\sigma_a^f)^{-1}(b)) = 
-f(b) + a \circ(\sigma_a^f)^{-1}(b)=-f(b)+f(a) +b = b\triangleright a.$
\item $~\sigma^f_a(b) \triangleright \sigma^f_a(c)= -f(a\circ b) + f(a\circ c) -f(a) +a\circ b~$ and
\[\sigma^f_a(b\triangleright_s c) = -f(a)+a-a \circ f(b) + a \circ f(c) -a +a\circ b.\] 
But via condition (3)  we conclude that 
$~\sigma^f_a(b) \triangleright_s \sigma^f_a(c)= \sigma^f_a(b \triangleright_s c).$
 
\item $~\sigma^f_a(\sigma^f_b(c)) =-f(a) +a - a\circ f(b) + a\circ b \circ c $ and via condition (3):
\[ \sigma^f_a(\sigma^f_b(c)) = -f(a\circ b) + a \circ b \circ c = \sigma^f_{a\circ b}(c).\]

Also,  $\tau^f_c(\tau^f_b(a)) = \sigma^f_{\tau^f_b(a)}(c)^{-1}\circ \sigma^f_a(b)^{-1} \circ a \circ b \circ c.$ 
Thus, it suffices to show that $\sigma_a^f(b) \circ \sigma^f_{\tau^f_b(a)}(c)=  \sigma_a^{f^2}(b\circ c). $ Indeed,
\[\sigma_a^f(b) \circ \sigma^f_{\tau^f_b(a)}(c) = \sigma_a^f(b) \circ ( -f(\tau_b^f(a)) +  \tau^f_b(a) \circ c) =\sigma_a^f(b) -\sigma_a^f(b) \circ f(\tau_b^f(a)) + a\circ b \circ c, \]
and via condition (3) we conclude that  $\sigma_a^f(b) \circ \sigma^f_{\tau_b(a)}(c)=  \sigma_a^{f^2}(b\circ c) \Rightarrow \tau^f_c(\tau_b^f(a)) =\tau^{f^2}_{b\circ c}(a). $
\item $~\sigma^f_a(b) \bullet_s \sigma^f_a(c) = f(\sigma^f_a(b) + \sigma^f_a(c)) = $
\[f(-f(a) + a\circ b)  -f(a) +a\circ c = -f^2(a) + f(a\circ b) - f(a) + a\circ c \] and by means of condition (3) we arrive at $~\sigma^f_a(b) \bullet_s \sigma^f_a(c) = \sigma^{f^2}_a(b \bullet_s c).$ 
\end{enumerate}

Similarly for part (ii). 
\end{proof}

\begin{cor} \label{rr2}
Let $\sigma_a^f, \ \tau^f_b: X \to X$ as derived in Proposition \ref{rr1}, and  $r: X\times X \to X\times X,$ 
such that
for all $a,b \in X,$ $r(a,b) = (\sigma^f_a(b), \tau^f_b(a)).$ Then $r$ is a solution of the set-theoretic braid equation. 
\end{cor}
\begin{proof}
This immediately follows from Proposition \ref{rr1}, specifically from part {\it d} of {\it (i)} and \it{(ii)}.
\end{proof}

\begin{lemma} 
Let $\sigma_a^f, \ \tau^f_b: X \to X$ as derived in Proposition \ref{rr1},  and $r: X\times X \to X\times X,$ such that
for all $a,b \in X,$ $r(a,b) = (\sigma^f_a(b), \tau^f_b(a))$ is a solution of the set-theoretic braid equation.  Then $r^{-1}(x,y) =(\hat \sigma^f_a(b),  \hat \tau^f_b(a)),$ where 
\begin{enumerate}[{(i)}]
\item  $~\hat \sigma^f_a(b) \circ \hat \tau^f_b(a) = a \circ b,$ $~\hat \sigma^f_a(b) = f^{-1}(a\circ b -a).$
\item   $~\hat \sigma^f_a(b) \circ \hat \tau^f_b(a) = a \circ b, $ $~\hat \sigma^f_a(b) = f^{-1}(-a +a\circ b).$
\end{enumerate}
\end{lemma}
\begin{proof} The proof is direct by means of the relations $~\sigma^f_{\hat \sigma^f_a(b)}(\hat \tau^f_b(a))=a~$ 
and $~\tau^f_{\hat \tau^f_b(a)}(\hat \sigma^f_a(b)) = b.$
\end{proof}

\begin{exa} We provide below two simple examples. Let $(X,+,\circ)$ be a skew brace and:
\begin{enumerate}
\item Let $f(a)=a,$ for all $a \in X,$ then 
(i) $~\sigma_a(b) = - a + a\circ b,$ (ii) $~\sigma_a(b) = a\circ b - a.$
\item Let $f(a) = a\circ z -z, $ $f'(a)= -z+a\circ z$ for all $a \in X,$ then  (i) $~\sigma^f_a(b) = z- a \circ z + a\circ b,$ (ii) $~\sigma^{f'}_a(b) = a\circ b-a\circ z+z.$  
\end{enumerate}
\end{exa}

\subsection*{Acknowledgments}
\noindent 
Support from the EPSRC research grant  EP/V008129/1 is  acknowledged.  
This work was partially supported by the Dipartimento di Matematica e Fisica 
“Ennio De Giorgi”,  Universit\`a del Salento.  PS is a member of GNSAGA (INdAM) and the non-profit
association ADV-AGTA.

\end{document}